\documentclass[12pt]{amsart}
\usepackage{amsmath,amssymb,amsthm}
\usepackage[dvips]{graphicx}
\usepackage{color}
\usepackage{wasysym}
\newtheorem{lem}{Lemma}[section]
\newtheorem{prop}{Proposition}[section]
\newtheorem{cor}{Corollary}[section]

\newtheorem{thm}{Theorem}[section]

\theoremstyle{definition}
\newtheorem{definition}{Definition}[section]

\theoremstyle{remark}

\theoremstyle{remark}

\numberwithin{equation}{section}

\newcommand{\R}{{\mathbb R}}

\definecolor{blu}{rgb}{0,0,1}
\def\im{{\rm i}}

\newcommand{\vertiii}[1]{{\left\vert\kern-0.25ex\left\vert\kern-0.25ex\left\vert #1
    \right\vert\kern-0.25ex\right\vert\kern-0.25ex\right\vert}}

\begin{document}
\title[Semirelativistic NLS and  Half Wave  in arbitrary dimension]{Long time dynamics for semirelativistic NLS and  Half Wave  in arbitrary dimension}
\author{Jacopo Bellazzini}
\address{J. Bellazzini,
\newline  Universit\`a di Sassari, Via Piandanna 4, 70100 Sassari, Italy}%
\email{jbellazzini@uniss.it}%
\author{Vladimir Georgiev}
\address{V. Georgiev
\newline Dipartimento di Matematica Universit\`a di Pisa
Largo B. Pontecorvo 5, 56100 Pisa, Italy}%
\email{georgiev@dm.unipi.it}%
\author{Nicola Visciglia}
\address{N. Visciglia, \newline Dipartimento di Matematica Universit\`a di Pisa
Largo B. Pontecorvo 5, 56100 Pisa, Italy}%
\email{viscigli@dm.unipi.it}

\begin{abstract}
We consider the Cauchy problems associated with semirelativistc NLS (sNLS) and half wave (HW).
In particular we focus on the following two main questions: local/global Cauchy theory; existence and stability/instability of ground states.
In between other results, we prove the existence and stability of
ground states for sNLS  in the $L^2$ supercritical regime.
This is in sharp contrast with the instability of ground states for the
corresponding HW, which is also established along the paper, by showing an inflation of norms phenomenon.
Concerning the Cauchy theory we show, under radial symmetry assumption the following results: a local existence result
in $H^1$ for energy subcritical nonlinearity and  a global existence result
in the $L^2$ subcritical regime.

\end{abstract}

\maketitle

The aim of this paper is the analysis of the following Cauchy problems
with special emphasis to the local/global existence and uniqueness results, as well as
to the issue of existence and stability/instability of ground states:
\begin{equation}\label{evolution1}
\begin{cases}
\im \partial_t u =A u - u|u|^{p-1},\quad (t,x)\in \R\times \R^{n}\\
u(0, x)=f(x)\in H^s(\R^n),
\end{cases}
\end{equation}
where $A=\sqrt {-\Delta} \hbox{ and } A=\sqrt{1-\Delta},$ namely
Half Wave (HW)
and semirelativistic NLS (sNLS).
Since now on
$H^s(\R^n)$ and $\dot H^s(\R^n)$ denote respectively
the usual inhomogenoues and homogeneous
Sobolev spaces in $\R^n$, endowed with the norms $\|(1-\Delta)^{s/2} u\|_{L^2(\R^n)}$ and
$ \|(-\Delta)^{s/2} u\|_{L^2(\R^n)}$. We shall also refer
to $H^s_{rad}(\R^n)$ as to the set of functions belonging to $H^s(\R^n)$ which are radially symmetric.
\\

Along the paper we shall study several properties of the Cauchy problems associated with
sNLS and HW. The first result will concern the local/global Cauchy theory at
low regularity under an extra radiality assumption. We point out that at the best of our knowledge in the literature there exist very few results about the global existence of solutions to both HW and sNLS. In particular we mention the result in \cite{OV} where it is
considered HW in $1-d$ with nonlinearity $u|u|^3$ and initial data in $H^1(\R)$,
without any further symmetry assumption. Indeed the aforementioned result can be extended to $1-d$ sNLS with quartic nonlinearity. We also underline that in $1-d$ no results are available concerning the global existence for higher order nonlinearity, namely $p>4$.
One novelty in this paper is that we provide global existence results in higher dimension $n\geq 2$
under the radial symmetry assumption, provided that $p$ satisfies some restrictions.
\\

Another important issue considered along this article is the the existence and stability/instability properties of solitary waves associated with sNLS and HW. Of course, the first main ingredient
in order to speak about dynamical properties of the solitary waves, is
a robust Cauchy theory that at the best of our knowledge is provided in this paper for the first time in the radially symmetric setting.
\\

We recall that two values of the nonlinearity $p$ are quite relevant: the nonlinearity $u|u|^{2/(n-1)}$,
which is $H^{1/2}(\R^n)$-critical, and the nonlinearity
$u|u|^{2/n}$, which is $L^2(\R^n)$-critical.
Next we present our main result about the Cauchy problems \eqref{evolution1}:
we prove on one hand a local existence result in $H^1_{rad}(\R^n)$
via contraction argument for $H^{1/2}(\R^n)$ subcritical nonlinearity; on the other hand we
show that the solutions are global in time provided that the nonlinearity is $L^2(\R^n)$-subcritical and
we assume an a-priori bound on $H^{1/2}(\R^n)$ norm of the solution.

\begin{thm} \label{t.le0} Let $n\geq 2$, $A$ be either $\sqrt{-\Delta}$ or $\sqrt{1-\Delta}$,
$p \in (1,1+\frac 2{n-1})$. Then for every $R>0$ there exists $T=T(R)>0$
and a Banach space $X_T$ such that:
\begin{itemize}
\item $X_T\subset {\mathcal C}([0,T];H^{1}_{rad}(\R^n))$;
 \item
for any $f(x)\in H^{1}_{rad}(\R^n)$ with
$ \|f\|_{H^{1}(\R^n)} \leq R$,
there exists a unique solution  $u(t,x) \in X_T $ of \eqref{evolution1}.
\end{itemize}
Assume moreover that  $p\in(1,1+\frac 2n)$,
then  the solution is global in time.
\end{thm}

We point out that by a cheap argument (based only on Sobolev embedding and energy estimates) one can solve locally in time the
Cauchy problem \eqref{evolution1} for initial data  $f(x)\in H^{n/2+\epsilon}(\R^n)$ (without any radiality assumption). Notice that
we provide a local existence result, in radial symmetry,
with regularity $H^1(\R^n)$ for $n\geq 2$. Indeed
it will be clear to the reader, by looking at the proof of Theorem \ref{t.le0}, that
one can push the local theory at
the level of regularity $H^{1/2+\epsilon}_{rad}(\R^n)$. The main technical difficulty
to go from $H^1_{rad}(\R^n)$ to $H_{rad}^{1/2+\epsilon}(\R^n)$ being the fact that
in the first case we work with straight derivatives, and hence
the weighted chain rules that we need along the proof are straightforward.
In the second case the proof requires more delicate commutator estimates that we prefer to skip along this paper.
We also underline that in the $L^2$ subcritical regime we get a global existence result.
\\
\\

Next we shall analyze the issue of standing waves. We recall that standing waves are special solutions
to \eqref{evolution1} with a special structure, namely
$u(t, x)=e^{i\omega t} v(x)$, where $\omega\in \R$ plays the role of the frequency.
Indeed $u(t,x)$ is a standing wave solution if and only if $v(x)$ satisfies
\begin{equation}\label{solitwaves}A v + \omega v - v|v|^p=0\quad
\hbox{ in } \R^n.\end{equation} It is worth mentioning that, following the pioneering paper \cite{CL},
it is well understood how to build up solitary waves for both sNLS and HW
via a energy constrained minimization argument, in the case of $L^2$-subcritical nonlinearity.
Moreover as a byproduct of this variational approach, the corresponding solitary waves are   orbitally stable.
We recall that sNLS and HW enjoy respectively
the conservation of the following energy:
\begin{equation}\label{eq:varprob}
{\mathcal E}_{s}(u)=\frac{1}{2}\|u\|_{H^{1/2}(\R^n)}^2-\frac{1}{p+1}\| u \|^{p+1}_{L^{p+1}(\R^n)},
\end{equation}
\begin{equation}\label{eq:varprob2}
{\mathcal E}_{hw}(u)=\frac{1}{2}\|u\|_{\dot H^{1/2}(\R^n)}^2+\frac{1}{2}\|u\|_{2}^2-\frac{1}{p+1}\| u  \|^{p+1}_{L^{p+1}(\R^n)},
\end{equation}
as well as the conservation of the mass, namely:
\begin{equation}\label{masscons} \frac d{dt} \|u(t, x)\|_{L^2(\R^n)}^2=0
\end{equation}
for solutions $u(t,x)$ associated with \eqref{evolution1}.
In the nonlocal context in which we are interested in, the minimization problems analogue of the one studied in
\cite{CL} for NLS are the following ones:
\begin{equation}\label{Jsub}{\mathcal J}^s_r=\inf_{u\in S_r} {\mathcal E}_s(u), \quad \quad
{\mathcal J}^{hw}_r=\inf_{u\in S_r}  {\mathcal E}_{hw}(u)\end{equation}
where
\begin{equation}\label{Srsub}S_r=\big \{u\in H^{1/2}(\R^n) \hbox{ s.t. } \|u\|_{L^2(\R^n)}^2=r\big \}.
\end{equation}
Indeed it is not difficult (following the rather classical concentration-compactness argument,
see for instance \cite{BFV} for more details in the non-local setting)
to get a strong compactness property (up to translation)
for minimizing sequences associated with the minimization problems above,
provided that the nonlinearity is $L^2$ subcritical, i.e. $1<p<1+\frac 2n$.
By combining this fact
with the global existence result
stated in Theorem \ref{t.le0}, one can prove a stability result that we state below.
In order to do that first we need to introduce a suitable notion of stability,
that is weaker respect to the usual one.
This is
due mainly to the fact that we are not able to get
any  global existence result for the Cauchy problem associated with sNLS and HW at the level of regularity of the Hamiltonian $H^{1/2}$ and without the radiality assumption.
Hence we need to
assume more regularity and also the radial symmetry on the perturbations allowed
along the definition of stability.
\begin{definition}\label{wos}
Let ${\mathcal N}\subset H^{1}_{rad}(\R^n)$ be bounded in $H^{1/2}(\R^n)$.
We say that $\mathcal N$
is {\em weakly orbitally stable} by the flow associated with sNLS (resp. HW) if for any $\epsilon>0$ there exists $\delta>0$
such that
$$dist_{H^{1/2}}(u(0,.), {\mathcal N})<\delta \hbox{ and } u(0,x)\in H^{1}_{rad}(\R^n)
\Rightarrow
$$$$ \Phi_t(u(0,.)) \hbox{ is globally defined and }
\sup_t dist_{H^{1/2}}(\Phi_t(u(0,.)), {\mathcal N})<\epsilon$$
where $dist_{H^{1/2}}$ denotes the usual distance
with respect to the topology of $H^{1/2}$
and $\Phi_t(u(0,.)$ is the unique global solution associated with the Cauchy problem sNLS
(resp. HW)
and with initial condition $u(0,x)$.
\end{definition}
We can now state the next result, where we use the notations \eqref{Jsub} and \eqref{Srsub}. We state it as a corollary since it is a classical consequence of the concentration-compactness argument in the spirit of \cite{CL} and Theorem \ref{t.le0}, that guarantees a global dynamic for sNLS and HW. Hence we shall not provide the straightforward
proof along the paper. Neverthless we believe that it has its own interest.
\begin{cor}\label{semirstabilsubcrit} Let $1<p<1+\frac 2{n}$ and $n\geq 1$.
Then for every $r>0$ we have:
\begin{itemize}
\item ${\mathcal J}_r^s>-\infty$ (resp. ${\mathcal J}_r^{hw}>-\infty$) and
${\mathcal B}_r^s \neq \emptyset$ (resp. ${\mathcal B}_r^{hw}\neq \emptyset$) where
$${\mathcal B}_r^s:=\{v\in S_r \hbox{ s.t. }{\mathcal E}_{s}(v)={\mathcal J}_r^s\}$$
(resp. ${\mathcal B}_r^{hw}:=\{v\in S_r \hbox{ s.t. }{\mathcal E}_{hw}(v)={\mathcal J}_r^{hw}\}$).
In particular for every $v\in {\mathcal B}_r^s$ (resp.
$v\in {\mathcal B}_r^{hw}$) there exists
$\omega\in \R$ such that $$\sqrt{1-\Delta}  v + \omega v-v|v|^{p-1}=0$$
(resp. $\sqrt{-\Delta} v + \omega v-v|v|^{p-1}=0$);
\item the set ${\mathcal B}_r^s$ (resp. ${\mathcal B}_r^{hw}$)
is {\em weakly orbitally stable} by the flow associated with sNLS (resp. HW).
Moreover in the case $n=1$ the {\em weak orbital stability} property
can be strengthened, in the sense that in the Definition \ref{wos}
we can replace $H^1_{rad}(\R)$ by the larger space $H^1(\R)$.

\end{itemize}
\end{cor}

On the contrary, the situation dramatically changes in the $L^2$-supercritical regime
(namely $p>1+\frac 2{n}$) since the aforementioned minimization problems \eqref{Jsub} are meaningless, in the sense
that: \begin{equation}\label{...}{\mathcal J}_r^s={\mathcal J}_r^{hw}=-\infty, \quad \forall r>0, \quad p>1+\frac 2{n}.\end{equation}

Next result is aimed to show a special geometry (local minima) for the constrained energy
associated to sNLS
in the $L^2$-supercritical regime, i.e $1+\frac 2n<p<1+\frac 2{n-1}$.
In order to state our next result let us first introduce a family of localized and
constrained minimization problems:
\begin{equation}\label{Jr}{\mathcal J}_r=\inf_{u\in S_r\cap B_1} {\mathcal E}_s (u),
\end{equation}
where
$$B_\rho=\{u\in H^{1/2}(\R^n) \hbox{ s.t. } \|(1-\Delta)^{\frac 14} u\|_{L^2(\R^n)}\leq \rho\}.$$
We also recall that the notion of {\em weak orbital stability} is given in Definition \ref{wos}.
\begin{thm}\label{semir} Let $1+\frac 2n<p<1+\frac 2{n-1}$ and $n\geq 1$.
There exists $r_0>0$ such that the following conditions occur
for every $r\in (0, r_0)$:
\begin{itemize}
\item ${\mathcal J}_r>-\infty$,
${\mathcal B}_r \neq \emptyset$ and ${\mathcal B}_r\subset B_{1/2}\cap H^1(\R^n)$, where
$${\mathcal B}_r:=\{v\in S_r\cap B_1 \hbox{ s.t. }{\mathcal E}_{s}(v)={\mathcal J}_r\}.$$
In particular for every $v\in {\mathcal B}_r$ there exists
$\omega\in \R$ such that $$\sqrt{1-\Delta} v + \omega v-v|v|^{p-1}=0;$$
\item the elements in ${\mathcal B}_r$ are ground states on $S_r$, namely:
$$\inf_{{\mathcal C}_r} {\mathcal E}_s(w)
={\mathcal J}_r \hbox{ where }
{\mathcal C_r}=\{w\in S_r\hbox{ s.t. } {{\mathcal E}_s'}|_{S_r}(w)=0\}.$$
\end{itemize}
Assume moreover the following assumption:
\begin{equation}\label{**}\sup_{(-T_-(f), T_+(f))} \|u(t, x)\|_{H^{1/2}(\R^n)}<\infty \Rightarrow
T_\pm(f) =\infty\end{equation}
where $(-T_- (f), T_+ (f))$
is the maximal time of existence of $u(t,x)$ which is the nonlinear solution to sNLS
with initial datum $f(x)\in H^1_{rad}(\R^n)$.
Then we get:
\begin{itemize}
\item the set ${\mathcal B}_r\cap H^{1/2}_{rad}(\R^n)$
is {\em weakly orbitally stable} for the flow associated with sNLS.
\end{itemize}
\end{thm}

We point out that
the extra assumption \eqref{**}
it is satisfied
for $n=1$ and $p=4$ without any radiality assumption (it follows by a suitable adaptation to sNLS of the argument given in \cite{OV} for HW).
An alternative and simpler argument for the global existence of $1-d$
quartic sNLS is given in the Appendix. Hence the statement above
provides the existence of stable standing waves for the quartic $1-d$ sNLS,
by removing the condition \eqref{**}.

More precisely we can state the following result.

\begin{cor}\label{corwos}
Let $n=1$ and $p=4$. Then
under the same notations as in Theorem \ref{semir}
we have that for $r<r_0$ the corresponding set ${\mathcal B}_r$ is
{\em weakly orbitally stable}. Indeed ${\mathcal B}_r$ satisfies
a straightened version of the property given in Definition \ref{wos}, where we can replace
$H^1_{rad}(\R)$ by $H^1(\R)$.
\end{cor}

We point out that the {\em weak orbital stability} stated in Theorem \ref{semir} under the condition \eqref{**},
as well as in Corollary \ref{corwos},
is a byproduct of the general Cazenave-Lions strategy (see \cite{CL}), once
the following compactness property (where we don't assume any radiality assumption)
is established:
$$u_k\in S_r\cap B_1, \quad {\mathcal E}_s (u_k) \rightarrow {\mathcal J}_r
\Rightarrow \exists x_k\in \R^n \hbox{ s.t. }$$$$
u_k(x+x_k)\hbox{ has a strong limit in } H^{1/2}(\R^n).$$
The main difficulty here being the fact that we have to deal with a
local minimization problems (since the global minimization problem is meaningless, see
\eqref{...}) and hence the application of the concentration-compactness
argument is much more delicate.
We also underline that if we look at the same minimization problems as above,
under the extra radiality assumption (namely $u_k(x)=u_k(|x|)$, then the compactness stated above occurs without any selection of the translation parameters $x_k$.

We would like to mention that at the best of our knowledge this is the first result
about translation invariant equations, where stable solitary waves are proved to exist
in the $L^2$-supercritical regime.
\\
In order to state our last result about existence/instability of ground states for HW,
we need to introduce also the following functional:
$${\mathcal P}(u)=\frac 12 \|u\|_{\dot H^{1/2}(\R^n)}^2- \frac{n(p-1)}{2(p+1)}\| u  \|^{p+1}_{L^{p+1}(\R^n)},$$
and the corresponding
set:
\begin{equation}\label{mathcalM}{\mathcal M}=\big \{u\in H^{1/2}(\R^n) \hbox{ s.t. }  {\mathcal P}(u)=0 \big \}.\end{equation}
It is well known (see \cite{secchi})  that we have the following inclusion
$$\{w\in S_r \hbox{ s.t. } {\mathcal E}_{hw}'|_{S_r}=0\}\subset \mathcal M,$$
namely every critical point of the energy ${\mathcal E}_{hw}$
on the constraint $S_r$ belongs to the set $\mathcal M$.
It is worth mentioning that this fact
is reminiscent of
the Pohozaev identity, which is here adapted to the case of HW.
The following minimization problem will be crucial in the sequel:
$${\mathcal I}_r=\inf_{S_r\cap {\mathcal M}} \mathcal E_{hw}(u).$$
\begin{thm}\label{HWar}
Let $n\geq 1$ and $1+\frac 2n<p<1+\frac 2{n-1}$. Then for every $r>0$
we have:
\begin{itemize}
\item ${\mathcal I}_r>-\infty$ and ${\mathcal A}_r \neq \emptyset$, where $${\mathcal A}_r:=\{v\in S_r\cap {\mathcal M}\hbox{ s.t. }  {\mathcal E}_{hw}(v)={\mathcal I}_r\}.$$
Moreover any $v\in {\mathcal A}_r$ satisfies
$$\sqrt{-\Delta} v + \omega v-v|v|^{p-1}=0$$
for a suitable $\omega\in \R$;
\item assume $f(x)\in S_r\cap H^1_{rad}(\R^n)$
satisfies ${\mathcal E}_{hw} (f)< {\mathcal I}_r$ and ${\mathcal P} (f)<0$, $n\geq 2$
and $u(t, x)$ is solution to \eqref{evolution1} (where $A=\sqrt{-\Delta}$),
 then the following alternative holds: either the solution blows-up in finite time or
$\|u(t,x)\|_{\dot H^{1/2}(\R^n)}\geq e^{a t}$
for suitable $a>0$. In particular the set ${\mathcal A}_r$
is not {\em weakly orbitally stable} for the flow associated with HW.
\end{itemize}
\end{thm}

Notice that in the first part of the statement, which is mostly variational, we don't assume the radial symmetry. On the contrary in the statement about the
evolution along the Cauchy problem we assume the radiality. This is mainly due to the fact that
at the best of our knowledge no global Cauchy theory is available
without the radiality assumption.
\\

We also underline that our approach  to prove the second part of Theorem \ref{HWar},
namely the norm inflation, is inspired by the work
of  Ogawa-Tsutsumi \cite{OZ} that was based, in the context of the classical NLS,
on the analysis of
time derivative of the localized virial
$$M_{\varphi_R}(u)=2 \text{ Im} \int \bar u \nabla \varphi_R \cdot \nabla u dx
$$
where $\varphi_R$ is a rescaled cut-off function such that $\nabla \varphi_R(x) \equiv x$ for $|x|\leq R$
and $\nabla \varphi_R(x) \equiv 0$ for $|x|>>R$.
This  approach has been further extended
by  Boulenger-Himmelsbach-Lenzmann \cite{BHL}
in the non-local context with dispersion $(-\Delta)^{s}$ for $\frac 12<s<1$.
In this paper we shall take advantage of similar computations
in the case of HW.
\\
We point out that the discrepancy between the dynamics for
sNLS and HW, revealed by  Theorems \ref{semir} and \ref{HWar}
about the stability/instability of ground states (namely ${\mathcal B}_r$
and ${\mathcal A}_r$) in the $L^2$ supercritical regime,
is reminiscent of the results of \cite{FL, FL2} for the dispersive equation describing a Boson Star:
\begin{equation*}\label{evolutionFL}
\im\partial_{t} u= \sqrt{m^2 -\Delta}u  -\left( \frac{1}{|x|}\star|u|^2\right)u,
\quad
(t,x)\in \R \times \R^{3}.
\end{equation*}
Indeed, in \cite{FL} it has been proved that ground states are unstable by blow up if $m=0$, while in \cite{FL2} it is shown that the ground states are orbitally stable whenever $m>0$.
However our situation is rather different from the one describing a Boson Star,
in fact in our case the constrained energy functional is always unbounded from below for any assigned $L^2$ constraint.
\\
\begin{figure}\label{fig2}
	\begin{center}
	{\includegraphics[width=8cm,height=5cm]{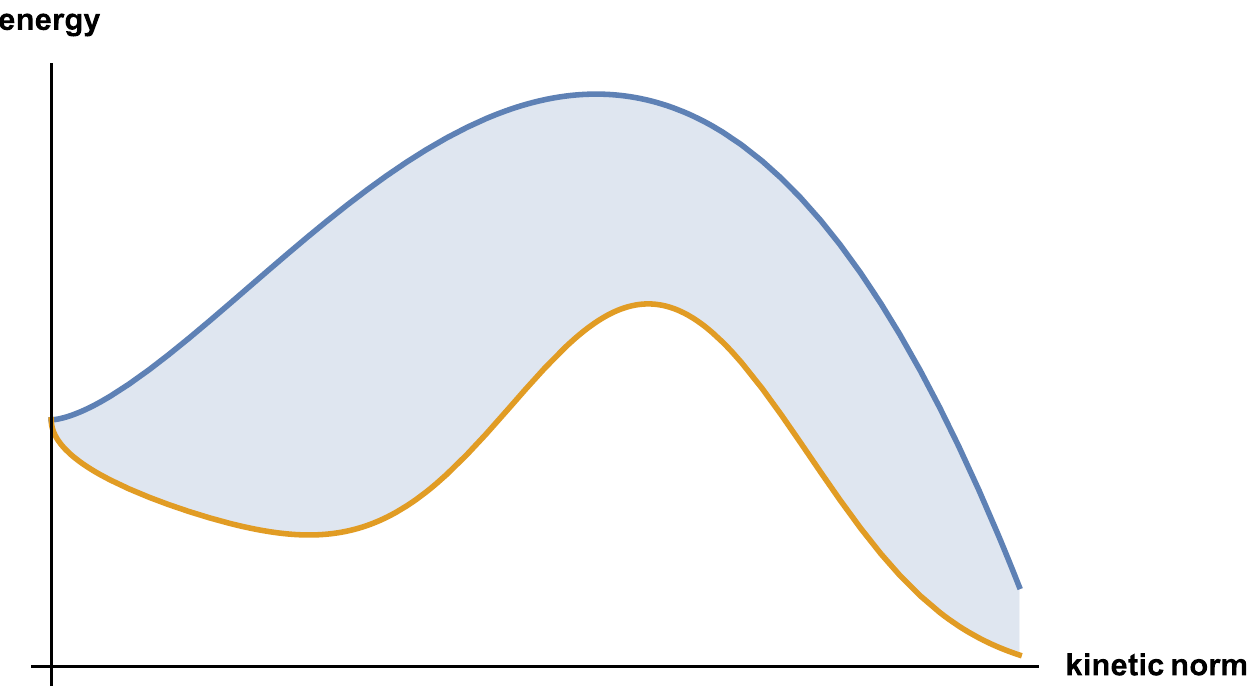}}
	\end{center}
\caption{Qualitative behavior of the constrained energy functionals associated with sNLS and HW.   In this qualitative picture the kinetic energy is given by  $\|u\|_{\dot H^{1/2}(\R^n)}$.}
\end{figure}	
\medskip
We conclude with a picture showing the main difference between
the functionals ${\mathcal E}_{hw}$ and ${\mathcal E}_{s}$ revealed
by Theorems \ref{semir} and \ref{HWar}, in the $L^2$ supercritical regime. In fact in the first case  we have established the stability of the ground states, and in
the second case we have proved on the contrary its instability.
For the HW (upper curve in the figure) the functional ${\mathcal E}_{hw}$ admits a critical point of mountain pass type. For  sNLS (lower curve in the figure) we have the existence of a
local minimizer for ${\mathcal E}_{s}$.


\section{The Cauchy theory for HW and sNLS}

The aim of this section it to prove a local/global existence and uniqueness result
for the Cauchy problem \eqref{evolution1}.
We need several tools that we shall exploit along the proof.
We treat in some details the result for the HW,
and we say at the end how to transfer the results
at the level of sNLS.
As usual we shall look for fixed point of the integral operator
associated with the Cauchy problem for HW:
\begin{equation}\label{eq.WF}
   S_f (u) = e^{-\im t \sqrt{-\Delta}} f +\im  \int_0^t  e^{-\im (t-\tau) \sqrt{-\Delta} }
  u(\tau)|u(\tau)|^{p-1}  d\tau.
\end{equation}
where $f(x)\in H^1_{rad}(\R^n)$.
We perform a fixed point argument in a suitable space
$X_T\subset C([0, T]; H^1_{rad}(\R^n))$
that, as we shall see below, is provided by an interpolation between a Kato-smoothing type estimate and the usual energy estimates.

\subsection{A Brezis-Gallou\"et-Strauss Type Inequality in $H^{1/2+\epsilon}_{rad}(\R^n)$}

In this subsection we introduce two functional inequalities
that will be useful respectively to achieve the local Cauchy theory and the globalization argument,
following in the spirit the paper by Strauss (see \cite{S77}) and
Brezis-Gallou\"et
(see \cite{BrGa}).
\begin{prop} \label{BG1l}
For every $n\geq 2$ and $s>1/2$ there exists a constant $C=C(s,n)>0$ such that:
\begin{equation}\label{trig}\||x|^{\frac{n-1}2} u\|_{L^\infty(\R^n)}\leq C \|u\|_{H^s(\R^n)};\end{equation}
\begin{equation}\label{brgal}
\||x|^{\frac{n-1}2} u\|_{L^\infty(\R^n)}\leq C \|u\|_{H^{1/2}(\R^n)}
\sqrt {\ln \left(2+ \frac{\|u\|_{H^s(\R^n)}}{\|u\|_{H^{1/2}(\R^n)}}\right)},\end{equation}
for every $u\in H^s_{rad}(\R^n)$.
\end{prop}
\begin{proof}
In the radial case it is well-known the Strauss estimate
(\cite{S77})
\begin{equation*}\label{eq.BGS}
   |x|^\frac{n-1}2 |u(x)| \leq C \|f\|_{H^1(\R^n)}, \quad \forall u\in H^1_{rad}(\R^n)
\end{equation*}
that has been extended in \cite{SS} to
\begin{equation}\label{eq.BGS1}
    |x|^\frac{n-1}2 |u_j(x)| \leq C \|u_j\|_{H^{1/2}(\R^n)}, \quad \forall u\in H^{1/2}_{rad}(\R^n), \quad \forall j\geq 0.
\end{equation}
Here we use the notation
$$ u_j = \varphi_j \left(\sqrt{-\Delta} \right) u, \quad \forall j \geq 0, $$
and $\varphi_j(s) $ is
the usual Paley-Littlewood decomposition, namely
$\varphi_j(s) \in C^\infty_0((0,\infty))$ are non-negative function supported in $[2^{j-1},2^{j+1}],$ such that:
$$ \sum_{j\geq 0} \varphi_j(s) = 1, \quad \forall s \geq 0.$$
Notice that the first estimate \eqref{trig} follows by decomposing
$u(x)=\sum_{j\geq 0} u_j(x)$ and by noticing that by Minkowski
inequality and \eqref{eq.BGS1}
$$\||x|^\frac{n-1}2 u\|_{L^\infty(\R^n)}\leq \sum_j \||x|^\frac{n-1}2 u_j\|_{L^\infty(\R^n)}
\leq C \sum_j \|u_j\|_{H^{1/2}(\R^n)}\leq C \|u\|_{H^s(\R^n)}$$
where at the last step we have used
the Cauchy-Schwartz inequality and the assumption $s>1/2$.

Concerning the proof of \eqref{brgal} we refine the argument above
as follows:

$$ \||x|^\frac{n-1}2  u\|_{L^\infty(\R^n)}   \leq C \sum_{j=0}^\infty  \|u_j\|_{H^{1/2}(\R^n)}
 = \underbrace{\sum_{j=0}^M  \|u_j\|_{H^{1/2}(\R^n)}}_{S_1(M)} + \underbrace{\sum_{j=M+1}^\infty  \|u_j\|_{H^{1/2}(\R^n)}}_{S_2(M)}$$ with $M$ being sufficiently large integer. We can estimate these two terms by Cauchy-Schwartz
as follows:
$$ S_1(M) \leq C \sqrt { M} \|u\|_{H^{1/2}(\R^n)},$$
$$ S_2(M)\leq C 2^{-M(1/4+s/2))} \|u\|_{H^{s}(\R^n)}$$
so we get
$$ \||x|^\frac{n-1}2 u(x)\|_\infty \leq C \sqrt{M} \|u\|_{H^{1/2}(\R^n)} + C 2^{-M(1/4+s/2)} \|u\|_{H^{s}(\R^n)}.$$ We conclude by choosing:
$$ M = \ln \big (2 + \frac{\|u\|_{H^{s}(\R^n)}}{\|u\|_{H^{1/2}(\R^n)}}\big ).$$

\end{proof}




\subsection{Energy Estimates and Kato Smoothing}

Next proposition is the key estimate for the linear propagator,
that will suggest
the space $X_T$ where to perform a fixed point argument.
In the sequel we shall use the notation
$$[x]_\delta=|x|^{1+\delta}+|x|^{1-\delta}.$$
\begin{prop}\label{freeev} Let $\delta>0$ be fixed.
We have the following bound
$$\|[x]_\delta^{-\frac 1q} e^{\im t \sqrt{-\Delta}} f\|_{L^q(\R; L^2(\R^n))}
\leq C \|f\|_{L^2(\R^n)}$$
for every $q\in [2, \infty]$ and $C>0$ is an universal constant that does not depend
on $q$.
\end{prop}
\begin{proof}
By interpolation it is sufficient to treat $q=2, \infty$.
The case $q=\infty$ is trivial and follows by the isometry
$\| e^{\im t \sqrt{-\Delta} }f\|_{L^2(\R^n)}=\|f\|_{L^2(\R^n)}$.\\
The case $q=2$ follows by combining the Kato smoothing (see \cite{kato},
\cite{RS})
together with the following lemma that provides uniform weighted estimates
for the resolvent associated with $\sqrt{-\Delta}$.

\begin{lem} Let $\delta>0$ be fixed, then
we have the following uniform bounds:
$$\|[x]_\delta^{-\frac 12} (\sqrt{-\Delta}-(\lambda+\im \epsilon))^{-1} f\|_{L^2(\R^n)}
\leq C \|  [x]_\delta^{\frac 12} f\|_{L^2(\R^n)}$$
where $C>0$ does not depend on
$\lambda, \epsilon>0$.
\end{lem}
\begin{proof}
We have the following identity
$$(\sqrt{-\Delta}-(\lambda+\im \epsilon))^{-1}= (\sqrt{-\Delta}+\lambda+\im \epsilon) \circ (-\Delta-(\lambda+\im \epsilon)^2)^{-1}$$
and hence
the desired estimate follows by the following well-known estimates
available for the resolvent associated with the
Laplacian operator $-\Delta$ (see \cite{agmon}, \cite{RS4}):$$\|[x]_\delta^{-\frac 12} \sqrt{-\Delta} (-\Delta-(\lambda+\im \epsilon)^2)^{-1}
f\|_{L^2(\R^n)}
\leq C \|  [x]_\delta^{\frac 12} f\|_{L^2(\R^n)}$$
and
$$\|[x]_\delta^{-\frac 12} (-\Delta-(\lambda+\im \epsilon)^2)^{-1} f\|_{L^2(\R^n)}
\leq \frac{C}{|\lambda+\im \epsilon|} \|  [x]_\delta^{\frac 12} f\|_{L^2(\R^n)}.$$
\end{proof}
The proof of Proposition \ref{freeev} is complete.

\end{proof}
Next we present a-priori estimates
associated with the Duhamel operator.
\begin{prop}\label{duham} Let $\delta>0$ be fixed.
For every $q_1\in [2, \infty]$ and $q_2\in (2, \infty]$ we get
$$\|[x]_\delta^{-\frac 1{q_1}} \int_0^t e^{\im (t-s)\sqrt{-\Delta}} F(s) ds\|_{L^{q_1}(\R;L^2(\R^n))}
\leq C \|[x]_\delta^{\frac 1{q_2}} F\|_{L^{q_2'}(\R;L^2(\R^n))}.$$
\end{prop}

\begin{proof}
The proof follows by combining Proposition \ref{freeev}
with the $T T^*$ argument (see \cite{GV92})
in conjunction
with the Christ-Kiselev Lemma (see \cite{christ}).

More precisely let
$T$ be the following operator:
$$T:L^2(\R^n)\ni f\rightarrow e^{\im t\sqrt{-\Delta}} f\in X_q$$
where  $$\|G(t,x)\|_{X_q}
=\|[x]_\delta^{-\frac 1q} G(t,x)\|_{L^q(\R; L^2(\R^n))}.
$$
Notice that $T$ is continuous by Proposition \ref{freeev}, and hence
by a duality argument we get the continuity of the operator
$$T^*: Y_{q}\ni F(t,x)\rightarrow \int_\R e^{-\im s\sqrt{-\Delta}} F(s) ds\in L^2(\R^n).$$
(here we have used the dual norm $\|G(t,x)\|_{Y_q}
=\|[x]_\delta^{1/q} G(t,x)\|_{L^{q'}(\R; L^2(\R^n))}.
$)
As a consequence, by choosing respectively $q=q_1$ and $q=q_2$
in the estimates above, we deduce that the following operator is continuous:
$$T\circ T^*: Y_{q_2}\ni F(t,x)\rightarrow \int_\R e^{\im (t-s)\sqrt{-\Delta}} F(s) ds\in X_{q_1}$$
and hence
$$
\|\int_\R e^{\im (t-s)\sqrt{-\Delta}} F(s) ds\|_{X_{q_1}} \leq C \|F\|_{Y_{q_2}}.$$
In fact by a straightforward localization argument (namely choose $F(s,x)$
supported only for $s>0$) we get
$$
\|\int_0^\infty e^{\im (t-s)\sqrt{-\Delta}} F(s) ds\|_{X_{q_1}} \leq C \|F\|_{Y_{q_2}}.$$
Notice that this estimate looks very much like the one that we want to prove, except
that we would like to replace the integral
$\int_0^\infty$ by the truncated integral $\int_0^t$.
This is possible thanks to the general
Christ-Kiselev Lemma mentioned above,
that works provided that $q_2'<q_1$.

\end{proof}

\subsection{Local Cauchy Theory in $H^{1}_{rad}(\R^d)$}
\label{sbsl}
We define the space $X_T$ and we perform in $X_T$ a contraction argument
for the integral operator $S_f$ (see \eqref{eq.WF}).

We introduce $q,\bar q>2$ and $\delta>0$ such that
\begin{equation}\label{impo}
-\frac{(n-1)(p-1)}{2}+\frac{1-\delta}{\bar q}=\frac{-1+\delta}q
\end{equation}
where $u|u|^{p-1}$ is the nonlinearity.

Notice that $q,\bar q, \delta$ as above exist provided that $p\in (1,1+\frac{2}{n-1})$.
Next we introduce
the space $X_T$ whose norm is defined as
\begin{align}\label{functspace}
\|u\|_{X_T}=\|u(t,x)\|_{L^\infty_T H^1(\R^n)}&+\|[x]_\delta^{-\frac 1q} \nabla_x
u(t,x)\|_{L^q_T L^2(\R^n)}\\\nonumber  &+\|[x]_\delta^{-\frac 1q}
u(t,x)\|_{L^q_T L^2(\R^n)}\end{align}
(here and below we use he notation $L^r_T(X)=L^r((0, T);X)$).
Next we introduce
a cut-off function $\psi\in C^\infty_c(\R^n)$ with $\psi(x)=0$ for $|x|>2$ and
$\psi(x)=1$ for $|x|<1$ and we write the forcing term $u|u|^{p-1}=\psi
u|u|^{p-1}+ (1-\psi)u|u|^{p-1}$.
Then we have, by using  Proposition
\ref{freeev} and Proposition \ref{duham} where we choose
$(q_1, q_2)=(\infty, \bar q)$
and $(q_1, q_2)=(q, \bar q)$ and where we apply the operator $\nabla_x$ (that
commutes with the equation):
\begin{align*}\|S_f u\|_{X_T}&\leq C \|f\|_{H^1(\R^n)}
+ C \|[x]_\delta^{\frac 1{\bar q}}\nabla_x (\psi (u|u|^{p-1}))\|_{L^{\bar q'}_T L^2(\R^n)}
\\&+ C \|\nabla_x ((1-\psi) (u|u|^{p-1}))\|_{L^{1}_T L^2(\R^n)}\end{align*}
where $S_f$ is the integral equation defined in \eqref{eq.WF}. Then we get
by using the Leibnitz rule and the properies of the cut-off function $\psi$:
\begin{align}\label{globalBr}\|S_f u\|_{X_T}& \leq C \|f\|_{H^1}
+ C \||x|^{-\frac{(n-1)(p-1)}{2}} |x|^{\frac{1-\delta}{\bar q}} (|x|^\frac{n-1}2 |u|)^{p-1} \nabla_x  u \|_{L^{\bar q'}_T L^2(|x|<2)}
\\\nonumber &+ C \||x|^{\frac{-1+\delta}{q}} (|x|^{\frac{n-1}2}
|u|)^{p-1} \|_{L^{\bar q'}_T L^2(|x|<2)} + C \| |u|^{p-1} \nabla_x u\|_{L^{1}_T L^2(|x|>1)}.
\end{align}
Next notice that by \eqref{impo}, we have
the estimate $\||x|^{\frac{n-1}2}u\|_{L^\infty(\R^d)}\leq C\|u\|_{H^1(\R^d)}$
for every $u\in H^1_{rad} (\R^d)$ and by H\"older in space and time
we get:
$$....\leq C \|f\|_{H^1(\R^d)}+ C T^{1-\frac{1}{\bar q}-\frac 1q} \|u\|_{L^\infty_T H^1(\R^d)}^{p-1}
\|u\|_{X_T}+ C T \|u\|_{L^\infty_T H^1(\R^d)}^{p}.$$
From this estimate one can conclude that
$$S_f: B_{X_T}(0, R)\rightarrow B_{X_T}(0, R)$$
for suitable $T, R>0$ where $$B_{X_T}(0, R)=\{v(t,x)\in X_T \hbox{ s.t. } \|v\|_{X_T}\leq R\}.$$
Next we endow the set $B_{X_T}(0, R)$ with the following distance:
$$d(u_1, u_2)=\|u_1-u_2\|_{L^\infty_T L^2(\R^n)}
+\|[x]_\delta^{-\frac 1q}
u(t,x)\|_{L^q_T L^2(\R^n)}.$$ It is easy to check that the metric space
$(B_{X_T}(0,R), d)$ is complete.
Then we conclude provided that we show that the map $S_f$
is a contraction on this space.
In order to do that we notice that by using the estimates in Proposition \ref{duham}
(but we don't apply in this case the operator $\nabla_x$) then we get:
$$\|S_f u_1(t) - S_fu_2(t)\|_{L^\infty_T L^2(\R^m)}
+\|[x]_\delta^{-\frac 1q}
(S_f u_1 - S_fu_2)\|_{L^q_T L^2(\R^n)}
$$$$
\leq C T^{1-\frac{1}{\bar q}-\frac 1q} (\|u_1\|_{X_T}^{p-1}+\|u_1\|_{X_T}^{p-1})
\|[x]_\delta^{-\frac 1q}(u_1-u_2)\|_{L^q_TL^2(\R^n)}
$$$$+ C T (\|u_1\|_{X_T}^{p-1}+\|u_1\|_{X_T}^{p-1})
 \|u_1-u_2\|_{L^\infty_T L^2(\R^m)} $$
Hence by choosing $T>0$ small enough and by recalling that
$u_1, u_2\in B_{X_T}(0,R)$ then we get:
$$d(S_f u_1, S_fu_2)\leq \frac 12 d(u_1, u_2), \quad \forall u_1, u_2\in B_{X_T}(0,R).$$
We conclude by using the contraction mapping principle.

\subsection{Conditional Global Existence in $H^{1}_{rad}(\R^n)$ for $1<p<1+\frac 2n$}
\label{ssgc1}$$$$
First notice that  for $1<p<1+\frac 2n$, then from Gagliardo-Nirenberg inequality we get
\begin{equation}\label{*}\sup_{(-T_-(f), T_+(f))} \|u(t, x)\|_{H^{1/2}(\R^n)}<\infty
\end{equation}
where $(-T_-(f), T_+(f))$ is the maximal time of existence.

By arguing as in the subsection \ref{sbsl}
(and by using the fact that $u(t,x)$ is a solution)
we get:
\begin{align*}\|u\|_{X_T}&\leq C \|f\|_{H^1(\R^n)}
+C \||x|^{-\frac{(n-1)(p-1)}2} |x|^{\frac{1-\delta}{\bar q}}\nabla_x u (|x|^\frac {n-1}{2} |u|)^{p-1})\|_{L^{\bar q'}_T L^2(|x|<2)}
\\\nonumber &+ C \||u|^{p-1} \nabla_x u \|_{L^{1}_T L^2(|x|>1)}.\end{align*}
By \eqref{impo} and H\"older in time we get:
\begin{align*}
\|u\|_{X_T}&\leq C \|f\|_{H^1(\R^n)}
\\\nonumber  +C \big((\||x|^{\frac{1-\delta}q}\nabla_x u \|_{L^{\bar q'}_T L^2(|x|<2)}
&+ T \|u \|_{L^{\infty}_T H^1(\R^n)}\big)
\||x|^{\frac{n-1}{2}}u\|_{L^\infty_T L^\infty(\R^n)}^{p-1}
\\\nonumber
\leq C \|f\|_{H^1(\R^n)}
+C (T^{1-\frac 1{\bar q} - \frac 1q} \|u \|_{X_T}
&+ T \|u \|_{L^{\infty}_T H^1(\R^n)})
\||x|^\frac{n-1}2u\|_{L^\infty_T L^\infty(\R^n)}^{p-1}.
\end{align*}
Hence if we introduce the function
 $g(T)= \sup_{t\in (0, T)} \|u\|_{X_t}$
 we obtain:
$$g(T)\leq C \|f\|_{H^1(\R^n)} + C \max\{T,
T^{1-\frac{1}{\bar q}-\frac 1q}\} g(T)
\||x|^\frac{n-1}{2}u\|_{L^\infty_T L^\infty(\R^n)}^{p-1}.$$
By using \eqref{brgal} in conjunction with the assumption
$\sup_t \|u(t,x)\|_{H^{1/2}(\R^n)}<\infty$,
we have:
\begin{equation}\label{duh+free}
g(T)\leq C \|f\|_{H^1(\R^n)} + C \max\{T,
T^{1-\frac{1}{\bar q}-\frac 1q}\} g(T) \ln^{\frac{p-1}2} (2+ Cg(T)).\end{equation}
Next we prove, as consequence of the estimate above, the following:
\\
\\
{\bf CLAIM}
{\em Let $\bar T>0$ \hbox{ be s. t. }
$C\max\{\bar T,
\bar T^{1-\frac{1}{\bar q}-\frac 1q}\}  \ln^{\frac{p-1}2} (2+ 2C^2\|f\|_{H^1(\R^n)})=\frac 12$
then $g(\bar T)\leq 2C \|f\|_{H^1(\R^n)}$.}
\\
\\
In order to prove the claim notice that if it is not true then
there exists $\tilde T <\Bar T$ such that $g(\tilde T)= 2C \|f\|_{H^1(\R^n)}$.
Then by going back to the proof of \eqref{duh+free}
and by using  the property $\tilde T <\Bar T$, one can prove:
\begin{align*}
g(\tilde T)\leq & C \|f\|_{H^1(\R^n)} + C \max\{\tilde T,
\tilde T^{1-\frac{1}{\bar q}-\frac 1q}\} g(\tilde T) \ln^{\frac{p-1}2} (2+ Cg(\tilde T))
\\\nonumber & < C \|f\|_{H^1(\R^n)} + C
\max\{\bar T,
\bar T^{1-\frac{1}{\bar q}-\frac 1q}\}  g(\tilde T) \ln^{\frac{p-1}2} (2+ 2C^2 \|f\|_{H^1(\R^n)})
\\\nonumber & = C \|f\|_{H^1(\R^n)} + \frac 12 g(\tilde T)
\end{align*}
and then we get
$g(\tilde T)< 2C \|f\|_{H^1(\R^n)}$, hence contradicting the definition of $\tilde T$.
\\
\\
By an iteration argument (based on the claim above) we can construct a sequence $\bar T_j$ such that
\begin{equation}\label{Tjbar}C\max\{\bar T_j,
\bar T_j^{1-\frac{1}{\bar q}-\frac 1q}\}  \ln^{\frac{p-1}2} (2+ 2C^2\|u(T_{j})\|_{H^1(\R^n)})=\frac 12,\end{equation}
and
\begin{equation}\label{gwp}
g(T_{j+1})\leq 2^{j+1} C^{j+1}\|f\|_{H^1(\R^n)},
\end{equation} where $T_{j+1}=\bar T_1+...+\bar T_j$.
We claim that $T_j\rightarrow \infty$ as $j\rightarrow \infty$, and in this case we conclude.
In fact if this is the case then
the solution can be extended to the interval
$[0, T_j]$ for every $j>0$
and of course it implies global well-posedness since $T_j\rightarrow \infty$.
Of course if there is a subsequence $T_{j_k}\geq 1$ then we conclude,
and hence it is not restrictive to assume $T_j<1$ at least for large $j$.
In particular we get
\begin{equation}\label{1q1brq}
\bar T_j^{1-\frac{1}{\bar q}-\frac 1q}= \max\{\bar T_j,
\bar T_j^{1-\frac{1}{\bar q}-\frac 1q}\}\sim
O( (\ln \|u(T_{j})\|_{H^1(\R^n)})^{-\frac{(p-1)}{2 }} ) \geq C (j^{-\frac{p-1}{2}})\end{equation}
where we used \eqref{Tjbar}
and the fact that \eqref{gwp} implies $\|u(T_{j})\|_{H^1(\R^n)}\leq 2^j C^{j}\|f\|_{H^1(\R^n)}$.
We conclude since by \eqref{impo} we can can choose $(q, \bar q)$ such that
$\frac{1}{\bar q}+\frac 1q=\frac{(n-1)(p-1)}2+\epsilon_0$ with $\epsilon_0>0$ arbitrarily small. In fact it
implies, together with \eqref{1q1brq}, the following estimate:
$$\bar T_{j_0+1}=\sum_{j=1}^{j_0} \bar T_j\geq C \sum_{j=1}^{j_0} j^{-\frac{p-1}{(2-(n-1)(p-1)-2\epsilon_0)}}$$
and the r.h.s. is divergent (for small $\epsilon_0>0$) provided that
$\frac{p-1}{2-(n-1)(p-1)}<1$, namely $p<1+\frac 2n$.

\subsection{Cauchy Theory for sNLS}
The idea is to reduce the Cauchy theory for sNLS to the Cauchy theory
for HW that has been established above.
More precisely let us introduce
the operator
$L=\sqrt{1-\Delta} - \sqrt{-\Delta}$
and hence we can rewrite the Cauchy problem associated with
sNLS as follows:
\begin{equation}\label{modif}
\begin{cases}
\im \partial_t u + \sqrt {-\Delta}=-L u+ u|u|^{p-1}, \quad (t, x)\in \R \times \R^n,\\
u(0,x)=f(x)\in H^{1}_{rad}(\R^n).
\end{cases}
\end{equation}
Notice that the operator $L$ corresponds in Fourier
at the multiplier $\frac{1}{\sqrt{1+|\xi|^2} + |\xi|}$,
and hence we have
$L:H^s(\R^n)\rightarrow H^{s}(\R^n)$.
Thanks to this property
it is easy to check that we can perform a fixed point argument
for \eqref{modif} in the space $X_T$ following the same argument
used to solve above for HW.
The minor change concerns the fact that
the extra term $Lu$ is absorbed in the nonlinear perturbation.
Also the globalization argument given in the subsection
\ref{ssgc1} can be easily adapted to  sNLS.


\section{Existence and stability of solitary waves for sNLS}

The main point along the proof of Theorem \ref{semir} is the proof of the compactness
(up to translation) of the minimizing sequences associated with $\mathcal J_r$, as well as the proof
of the fact that the minimizers belong to $B_{1/2}\cap S_r$,
provided that $r$ is small enough. This is sufficient in order to deduce that the minimizers
are far away from the boundary  and hence are constrained critical points.
In particular they satisfy the Euler -Lagrange equation up to
the Lagrange multiplier $\omega$.
Another delicate issue is to show
that the local minimizers (namely the elements in ${\mathcal B}_r$ according with the notation in Theorem \ref{semir})
have indeed minimal energy between all the critical points
of ${\mathcal E}_s$ constrained on the whole sphere $S_r$, provided that $r>0$ is small.

Next we shall focus on the points above, and we split the proofs in several steps.
We also mention that the statement about the orbital stability
it follows easily by the classical argument of Cazenave-Lions
(see \cite{CL}) once a nice Cauchy theory has been established.

\subsection{Local Minima Structure}

We start with the following lemma that shows a local minima structure for the functional
${\mathcal E}_s$ on the constraint $S_r$, for $r$ small enough.
\begin{prop}\label{locmin}
There exists $r_0>0$ such that:
\begin{equation}\label{lispetr}\inf_{\{u\in S_{r} | \|u\|_{H^{1/2}(\R^n)}=1\}}
{\mathcal E}_s(u)>\frac 14, \quad \forall r<r_0;\end{equation}
\begin{equation}\label{petrpet}\inf_{\{u\in S_{r}|\|u\|_{H^{1/2}(\R^n)}\leq 2 \sqrt r\}}
{\mathcal E}_s(u)
<\frac r2,
\quad \forall r<r_0.\end{equation}
\end{prop}
\begin{proof} By the Gagliardo-Nirenberg inequality we get for some $\epsilon_0>0$:
\begin{align}\label{importGN}
\frac 12 &\|u\|_{H^{1/2}(\R^n)}^2 -\frac 1{p+1}\|u\|_{L^{p+1}(\R^n)}^{p+1}
\geq \frac 12 \|u\|_{H^{1/2}(\R^n)}^2 - C_0 r^{\gamma_0} \|u\|_{H^{1/2}(\R^n)}^{2+\epsilon_0}
\\\nonumber &
= \frac 12 \|u\|_{H^{1/2}(\R^n)}^2 ( 1-C_0 r^{\gamma_0}
\|u\|_{H^{1/2}(\R^n)}^{\epsilon_0}),\quad \forall
u\in S_r
\end{align}
and hence
$$\frac 12 \|u\|_{H^{1/2}(\R^n)}^2 -\frac 1{p+1}\|u\|_{L^{p+1}(\R^n)}^{p+1}
> \frac 14,  \quad \forall
u\in S_r,\quad  \|u\|_{H^{1/2}(\R^n)}=1.
$$
Concerning the bound \eqref{petrpet} notice that:
\begin{align*}&\frac 12 \|u\|_{H^{1/2}(\R^n)}^2 -\frac 1{p+1}\|u\|_{L^{p+1}(\R^n)}^{p+1}
\\\nonumber  =\frac 12 \|u\|_{H^{1/2}(\R^n)}^2 & - \frac 12 \|u\|_{L^2(\R^n)}^2
+ \frac 12 \|u\|_{L^2(\R^n)}^2 -\frac 1{p+1}\|u\|_{L^{p+1}(\R^n)}^{p+1}
\end{align*}
and by using Plancharel
\begin{align*}...=\frac 12 \|u\|_{H^{1/2}(\R^n)}^2 - \frac 12 \|u\|_{L^2(\R^n)}^2
+ \frac r2  - \frac 1{p+1}\|u\|_{L^{p+1}(\R^n)}^{p+1}\\\nonumber
=
\frac 12 \|u\|_H^2  + \frac r2  -\frac 1{p+1}\|u\|_{L^{p+1}(\R^n)}^{p+1}, \quad \forall u\in H^{1/2}(\R^n)\cap S_r
\end{align*}
where
$$\|u\|_H^2=\int_{\R^n} \frac{|\xi|^2}{1+\sqrt {1+ |\xi|^2}} |\hat u(\xi)|^2
d\xi.$$
In particular we get
\begin{equation}\label{trunk}\|u\|_H^2\leq \frac 12 \|u\|_{\dot H^1(\R^n)}^2, \quad \forall u\in \dot H^1(\R^n)
\hbox{ s.t. } \hat u(\xi)=0, \quad \forall |\xi|>1.\end{equation}
Next we fix $\varphi$ smooth,
such that: $\hat \varphi(\xi)=0, \forall |\xi|\geq 1$ and $\|\varphi\|_{L^2(\R^n)}^2=r$.
We also
introduce
$\varphi_\lambda(x)$ where
$\hat \varphi_\lambda (\xi)= \lambda^{n/2} \hat \varphi( \lambda \xi),
$
then we get by the inequalities above (next we restrict to $\lambda>1$
in order to guarantee
$\hat \varphi_\lambda (\xi)=0, \forall |\xi|>1$ and hence we can apply \eqref{trunk}):
\begin{align*}
\frac 12 &\|\varphi_\lambda\|_{H^{1/2}(\R^n)}^2 -\frac 1{p+1}\|\varphi_\lambda\|_{L^{p+1}(\R^n)}^{p+1}
\\\nonumber
&\leq \frac 12 \|\varphi_\lambda\|_{\dot H^1(\R^n)}^2  + \frac r2  -\frac 1{p+1}\|\varphi_\lambda\|_{L^{p+1}(\R^n)}^{p+1}.
\end{align*}
Notice that $\varphi_\lambda\in S_r$. Moreover by a rescaling argument we get
$$\frac 12 \|\varphi_\lambda\|_{\dot H^1(\R^n)}^2  -\frac 1{p+1}\|\varphi_\lambda\|_{L^{p+1}(\R^n)}^{p+1}
<0$$
for any $\lambda$ large enough. We conclude since
$$\|\varphi_\lambda\|_{H^{1/2}(\R^n)}^2=\int_{\R^n} |\hat \varphi|^2 \sqrt{1+\frac{|\xi|^2}
{\lambda^2}} d\xi\rightarrow \|\varphi\|_{L^2(\R^n)}^2=r, \hbox{ as } \lambda \rightarrow \infty$$
and hence for $\lambda$ large enough $\|\varphi_\lambda\|_{H^{1/2}(\R^n)}<2\sqrt r$.

\end{proof}

\subsection{Avoiding Vanishing}

Next result will be crucial to exclude vanishing for the minimizing sequences.
In the sequel $r_0>0$ is the number that appears in Proposition
\ref{locmin}.\begin{prop}\label{vanis}
Assume $r<r_0$ and $u_k\in S_r\cap B_1$
be such that ${\mathcal E}_s(u_k)\rightarrow {\mathcal J}_r$
then $\liminf_{k\rightarrow \infty} \|u_k\|_{L^{p+1}(\R^n)}>0$.
\end{prop}

\begin{proof}
Assume by the absurd that it is false.
Then we get by Proposition  \ref{locmin}
$$\frac r2> {\mathcal J}_r=\lim_{k\rightarrow \infty} {\mathcal E}_s(u_k)
=\lim_{k\rightarrow \infty} \frac 12 \|u_k\|_{H^{1/2}(\R^n)}^2.$$
This is a contradiction since $u_k\in S_r$ and hence
$\|u_k\|_{H^{1/2}(\R^n)}^2\geq r$.

\end{proof}

\subsection{Avoiding Dichotomy}

Next result will be crucial to avoid dichotomy.

\begin{prop}\label{rich}
There exists $r_1>0$ such that for any $0<r<l<r_1$ we have
\begin{equation*}r {\mathcal J}_l<l {\mathcal J}_r.\end{equation*}
\end{prop}

We shall need the following result that allows us to get a bound on
the size of the minimizing sequences.
\begin{lem}\label{prelim}
There exists $r_2>0$ such that
$$\inf_{\{u\in S_{r}|\|u\|_{H^{1/2}(\R^n)}\leq 2\sqrt r\}}
{\mathcal E}_s (u)<
\inf_{\{u\in S_{r}|\sqrt r\leq \|u\|_{H^{1/2}(\R^n)}<1\}} {\mathcal E}_s (u),
\quad
\forall r<r_2.$$
\end{lem}
\begin{proof}
In view of Proposition  \ref{locmin}
it is sufficient to prove that
$$\inf_{\{u\in S_{r}|2 \sqrt r\leq \|u\|_{H^{1/2}(\R^n)}<1\}}
\big(\frac 12 \|u\|_{H^{1/2}(\R^n)}^2 -\frac 1{p+1}\|u\|_{L^{p+1}(\R^n)}^{p+1}\big)>\frac r2.$$
In order to prove it, we go back to
\eqref{importGN}
and we get
\begin{align*}
&\frac 12 \|u\|_{H^{1/2}(\R^n)}^2 -\frac 1{p+1}\|u\|_{L^{p+1}(\R^n)}^{p+1}
\\\nonumber
\geq \frac 12 \|u\|_{H^{1/2}(\R^n)}^2 &( 1-C_0 r^{\gamma_0} \|u\|_{H^{1/2}(\R^n)}^{\epsilon_0}),\quad \forall
u\in S_r
\end{align*}
and hence
$$...\geq  2 r ( 1-C_0 r^{\gamma_0}) , \quad \forall
u\in S_r, \quad 2\sqrt r \leq \|u\|_{H^{1/2}(\R^n)}<1.
$$
We conclude provided that $r$ is small enough.

\end{proof}

We can now conclude the proof of Proposition \ref{rich}.
Fix $v_k\in S_r, \|v_k\|_{H^{1/2}(\R^n)}\leq 1$  such that $\lim_{k\rightarrow \infty}
\mathcal{E}_s(v_k)={\mathcal J}_r$.

Notice that by Lemma \ref{prelim}
we can assume $\|v_k\|_{H^{1/2}(\R^n)}< 2 \sqrt r$.
In particular we have
$$\sqrt{\frac lr} v_k\in S_l \hbox{ and } \sqrt{\frac lr}
 \|v_k\|_{H^{1/2}(\R^n)}<2 \sqrt l,$$
and hence
$${\mathcal J}_l\leq \liminf_k
{\mathcal E}_s(\sqrt{\frac lr} v_k)$$$$= \frac 12 \frac{l}{r} \|v_k\|_{H^{1/2}(\R^n)}^2
- \frac 1{p+1} \frac{l^\frac{p+1}2}{r^\frac{p+1}2} \|v_k\|_{L^{p+1}(\R^n)}^{{p+1}}.$$
Recall that by Proposition \ref{vanis} we can assume
$$\|v_k\|_{L^{p+1}(\R^n)}>\delta_0>0$$
and hence we can continue the estimate above as follows
$$...=
 \frac{l}{r} \big( \frac 12
 \|v_k\|_{H^{1/2}(\R^n)}^2 - \frac 1{p+1} \|v_k\|_{L^{p+1}(\R^n)}^{p+1} \big)
+\frac 1{p+1}
\big (\frac{l}{r} - \frac{l^\frac{p+1}2}{r^\frac{p+1}2}\big) \|v_k\|_{p+1}^{p+1} $$
$$=\frac{l}{r}{\mathcal E}_s(v_k) + \big (\frac{l}{r} - \frac{l^\frac{p+1}2}{r^\frac{p+1}2}
\big ) \frac{\delta_0}{p+1}
\leq \frac{l}{r} {\mathcal J}_r  +\big (\frac{l}{r}- \frac{l^\frac{p+1}2}{r^\frac{p+1}2}\big ) \frac{\delta_0}{p+1}
<{\mathcal J}_r \frac{l}{r}.
$$
\subsection{Conclusions.}
Notice that by Proposition
\ref{locmin} we deduce that the minimizers (if exist)
have to belong necessarily to $B_{1/2}$. Next we prove
the compactness, up to
translations, of the minimizing sequences. Since now on we shall fix $r$
small enough
according with the Propositions above.

Let $u_k\in S_r$ be such that $\|u_k\|_{H^{1/2}(\R^n)}\leq 1$
and
${\mathcal E}_s(u_k)\rightarrow {\mathcal J}_r$, then
by combining Proposition \ref{vanis}
and with the Lieb translation Lemma in $H^{1/2}(\R^n)$ (see \cite{BFV}),  we have
that up to translation the weak limit of $u_k$ is $\bar u\neq 0.$
Our aim is to prove that $\bar u$ is  a strong limit in $L^2(\R^n)$. Hence
if we denote $\|\bar u\|_{L^2(\R^n)}^2=\bar r$ then it is sufficient to prove $\bar r=r$.
Notice that we have
by weak convergence
$$\|u_k - \bar u\|_{L^2(\R^n)}^2+\|\bar u\|_{L^2(\R^n)}^2=r+o(1)$$
and if we assume (by subsequence)
$\|u_k - \bar u\|_{L^2(\R^n)}^2\rightarrow t$ then we
have
$t+\bar r=r$. We shall prove that necessarily $t=0$ and hence $r=\bar r$.
Next by classical arguments, namely Brezis-Lieb Lemma
(see \cite{brlieb})
and the Hilbert structure of $H^{1/2}(\R^n)$, we get
$${\mathcal E}_s(u_k)={\mathcal E}_s(u_k-\bar u)+{\mathcal E}_s(\bar u)+o(1)\geq
{\mathcal J}_{\|u_k-\bar u\|_{L^2(\R^n)}^2} + {\mathcal J}_{\|\bar u\|_{L^2(\R^n)}^2}+o(1).
$$
Passing to the limit as $k\rightarrow \infty$, and by recalling $t+\bar r=r$ we get
\begin{equation}\label{addit}{\mathcal J}_{t+\bar r}={\mathcal J}_r\geq
{\mathcal J}_{\bar r} + {\mathcal J}_t.\end{equation}
On the other hand by
Proposition \ref{rich} we get
$$(t+\bar r){\mathcal J}_{\bar r}> \bar r {\mathcal J}_{t+\bar r} \hbox{ and } (t+\bar r)
{\mathcal J}_t> t {\mathcal J}_{t+\bar r}$$
that imply ${\mathcal J}_{\bar r}+ {\mathcal J}_{t}>{\mathcal J}_{t+\bar r}$ and it is in contradiction with
\eqref{addit}.
\\

As a last step we have to prove
that the local minima (namely the elements in ${\mathcal B}_r$)
minimize the energy ${\mathcal E}_s$ among all the
critical points of ${\mathcal E}_s$ constrained to $S_r$, provided that
$r>0$ is small enough.

In order to prove this fact we shall prove the following property:
$$\exists r_0>0 \hbox{ s.t. } \forall r<r_0 \hbox{ the following occurs }$$$$
\|w\|_{H^{1/2}(\R^n)} <1/2, \quad  \forall w\in S_r \hbox{ s.t. }  {\mathcal E}_s'|_{S_r}=0,
{\mathcal E}_s(w)<{\mathcal J}_r.$$

Once this fact is established then we can conclude easily since
it implies that if $w\in S_r$ is a constraint critical points with energy below ${\mathcal J}_r$
and $r<r_0$, then $w$  can be used as test functions to estimate ${\mathcal J}_r$
from above and we get
${\mathcal J}_r\leq  {\mathcal E}_s(w)$, hence we have a contradiction.

In order to prove the property stated above recall that if $w\in S_r$ is a critical point of
${\mathcal E}_s$ restricted to $S_r$, then notice that  it has to satisfy the
following Pohozaev type identity (this follows by an adaptation of the argument
in \cite{secchi} to sNLS):
\begin{equation*}
{\mathcal Q}(w)=0 \hbox{ where }
{\mathcal Q}(w)
=\frac 12 \int_{\R^n} \frac{|\xi|^2}{\sqrt{1+|\xi|^2}}|\hat w|^2d \xi-\frac{n(p-1)}{2(p+1)}
\| w\|^{p+1}_{L^{p+1}(\R^n)}
\end{equation*}
and hence
\begin{align*}\label{eq:esuV}{\mathcal E}_s(w)&=
{\mathcal E}_s(w)-\frac{2}{n(p-1)}{\mathcal Q}(w)\\\nonumber
& =\frac{1}{2}\|w\|_{H^{1/2}(\R^n)}^2- \frac{1}{n(p-1)}\int_{\R^n} \frac{|\xi|^2}{\sqrt{1+|\xi|^2}}|\hat w|^2d \xi\geq \frac{np-n-2}{2n(p-1)}\|w\|_{ H^{1/2}(\R^n)}^2.
\end{align*}
Notice that $np-n-2>0$ if $p>1+\frac2n$.
From the estimate above we get
$$\frac{np-n-2}{2n(p-1)}\|w\|_{\dot H^{1/2}(\R^n)}^2\leq {\mathcal E}_s(w)<{\mathcal J}_r<\frac{r}{2}$$
where we used Propostion \ref{locmin} at the last step.
It is now easy to conclude.

\section{Existence/instability of solitary waves for HW}

This section is devoted to the proof of Theorem \ref{HWar}.

\subsection{Existence of Minimizer}\label{minhw}

Even if Theorem \ref{HWar} is stated for the energy ${\mathcal E}_{hw}(u)$
on $S_r$ we shall work at the beginning on the unconstrained functional.
At the end we shall come back  to the constraint minimization
problem as stated in Theorem \ref{HWar}.

The first result concerns the fact that the constraint ${\mathcal M}$
(see \eqref{mathcalM})
is a natural constraint.

\begin{lem}\label{lagr}
Let $v\in H^{1/2}(\R^n)$ be such that
\begin{equation*}{\mathcal P}(v)=0,\quad {\mathcal E}_{hw}(v)
=\inf_{u\in {\mathcal M}} {\mathcal E}_{hw}(u)
\end{equation*}
then
$$\sqrt{-\Delta} v + v - v^p=0$$
\end{lem}

\begin{proof}

We notice that since $v$ is minimizer then
$$\sqrt{-\Delta} v + v - v^p=\lambda(\sqrt{-\Delta} v - \frac{n(p-1)}{2}v^p)$$
We claim that $\lambda=0$.
Notice that by the equation above we get
\begin{equation}\label{multiplication}
\|v\|_{\dot H^{1/2}(\R^n)}^2+\|v\|_2^2 - \|v\|_{L^{p+1}(\R^n)}^{p+1}=
\lambda \|v\|_{\dot H^{1/2}(\R^n)}^2 - \frac{d\lambda (p-1)}{2}
\|v\|_{L^{p+1}(\R^n)}^{p+1}
\end{equation}
Moreover since ${\mathcal P}(v)=0$ we get
\begin{equation}\label{pohhozaev}
\|v\|_{\dot H^{1/2}(\R^n)}^2 = \frac{n(p-1)}{(p+1)} \|v\|_{L^{p+1}(\R^n)}^{p+1}.
\end{equation}
Next notice that we have the following rescaling invariance
$${\mathcal P}(u)=0 \Rightarrow {\mathcal P}(\mu^{\frac 1{p-1}} u(\mu x)=0$$
and hence since $v$ is a minimizer we get
$\frac d{d\mu} {\mathcal E}_{hw}(\mu^{\frac 1{p-1}} v(\mu x))_{|\mu=1}=0$
that by elementary computations gives
\begin{align}\label{invariance}
\frac 12(\frac{p+1}{p-1}-n)&\|v\|_{\dot H^{1/2}(\R^n)}^2+\frac 12(\frac{2}{p-1}-n)\|v\|_{L^2(\R^n)}^2\\\nonumber&-(\frac{p+1}{p-1}-n)\frac{1}{p+1} \|v\|_{L^{p+1}(\R^n)}^{p+1}=0.\end{align}
By combining \eqref{multiplication},
\eqref{pohhozaev}, \eqref{invariance} we get easily
$\lambda=0$.

\end{proof}

The next result concerns the proof of the existence of a minimizer
for the unconstrained problem.

\begin{lem}\label{uncos} Let $1+\frac 2n<p<1+\frac 2{n-1}$.
There exists $w\in H^{1/2}_{rad}(\R^n)$ such that
$${\mathcal P}(w)=0, \quad {\mathcal E}_{hw}(w)=\inf_{u\in \mathcal M} {\mathcal E}_{hw}(u).$$
\end{lem}

\begin{proof}
For simplicity we denote $\inf_{u\in \mathcal M} {\mathcal E}_{hw}(u):={\mathcal E}_0$.

Notice that we have
\begin{align}\label{equivebergy}
{\mathcal E}_{hw}(u)&= (\frac{n(p-1)}{2(p+1)} - \frac 1{p+1})\|u\|_{L^{p+1}(\R^n)}^{p+1}
+\frac 12\|u\|_{L^2(\R^n)}^2, \\\nonumber
&\forall u\in H^{1/2}(\R^n), \quad {\mathcal P}(u)=0\end{align}
and hence, since $(\frac{n(p-1)}{2(p+1)} - \frac 1{p+1})>0$ for $p>1+\frac 2n$ we get
${\mathcal E}_0\geq 0.$
Let $w_k\in \mathcal M$ be a minimizing sequence.

We shall prove first the compactness of minimizing sequences
by assuming radial symmetry, namely $w_k\in H^{1/2}_{rad}(\R^n)$ such that
\begin{equation}\label{minim}
{\mathcal E}_{hw}(w_k)\rightarrow {\mathcal E}_0, \quad w_k\in H^{1/2}_{rad}(\R^n), \quad {\mathcal P}(w_k)=0.
\end{equation}
In a second step we shall prove that it is not restrictive to assume
that $w_n$ can be assumed radially symmetric.

First of all notice that we get $\sup_n \|w_k\|_{H^{1/2}(\R^n)}<\infty$.
In fact by using the constraint ${\mathcal P}(w_k)=0$ it is sufficient to check that
\begin{equation}\label{sopdel}
\sup_k \|w_k\|_{L^2(\R^n)} + \|w_k\|_{L^{p+1}(\R^n)}<\infty\end{equation}
and it follows by the expression \eqref{equivebergy} of the energy on the constraint ${\mathcal P}(u)=0$.\\
The next step is to show that $\inf_k \|w_k\|_{L^{p+1}(\R^n)}^{p+1}>0$.
It follows by the following chain of inequalities
\begin{align*}\|w_k\|_{H^{1/2}(\R^n)}^2&=\frac{n(p-1)}{p+1}\|w_k\|_{L^{p+1}(\R^n)}^{p+1}
\\\nonumber & \leq C \|w_k\|_{L^2(\R^n)}^{\gamma_0} \|w_k\|_{L^{p+1}(\R^n)}^{2+\epsilon_0}
\leq C \|w_k\|_{H^{1/2}(\R^n)}^{2+\epsilon_0} \end{align*}
where we used the Gagliardo-Nirenberg inequality, the boundedness of $\|w_k\|_{L^2(\R^n)}$
(see \eqref{sopdel}) and the Sobolev embedding
$H^{1/2}(\R^n)\subset L^{p+1}(\R^n)$.
As a consequence we get
$\inf_k \|w_k\|_{H^{1/2}(\R^n)}^2>0$ and since ${\mathcal P}(w_k)=0$ the same lower bound occurs
for $\inf_k \|w_k\|_{L^{p+1}(\R^n)}^{p+1}>0$.

Next we introduce $\bar w\in H^{1/2}(\R^n)$ as the weak limit of $w_k$.
We are done if we show that the convergence is strong.
First of all notice that by the compactness of the Sobolev embedding
$H^{1/2}_{rad}(\R^n)\rightarrow L^{p+1}(\R^n)$,
we deduce
$w_k\rightarrow \bar w$ in $L^{p+1}(\R^n)$ and hence
$\bar w\neq 0$ since $\inf_n \|w_k\|_{L^{p+1}(\R^n)}^{p+1}>0$.
Next notice that since ${\mathcal P}(w_k)=0$
then
$$\frac 12 \|\bar w\|_{\dot H^{1/2}(\R^n)}^2 - \frac{n(p-1)}{2(p+1)} \|\bar w\|_{L^{p+1}(\R^n)}^{p+1}\leq 0.$$
It implies by a continuity argument
$$\exists \bar \lambda\in (0, 1] \hbox{ s.t. } {\mathcal P}(\bar \lambda \bar w)=0,$$
and in turn
\begin{align*}&{\mathcal E}_0\leq {\mathcal E}_{hw}(\bar \lambda \bar w)\\
\nonumber&= (\frac{n(p-1)}{2(p+1)} - \frac 1{p+1})
\bar \lambda^{p+1} \|\bar w\|_{L^{p+1}(\R^n)}^{p+1}
+ \frac 12 \bar \lambda^2 \|\bar w\|_{L^2(\R^n)}^2\\\nonumber
&\leq \bar \lambda^2 ((\frac{n(p-1)}{2(p+1)} - \frac 1{p+1})
 \|\bar w\|_{L^{p+1}(\R^n)}^{p+1}
+\frac 12 \|\bar w\|_{L^2(\R^n)}^2)\\\nonumber & \leq (\frac{n(p-1)}{2(p+1)} - \frac 1{p+1})
\|\bar w\|_{L^{p+1}(\R^n)}^{p+1}
+\frac 12 \|\bar w\|_{L^2(\R^n)}^2 .\end{align*}
Notice that in the last inequality we get equality only in the case $\bar \lambda=1$.
Moreover by \eqref{equivebergy} and \eqref{minim} we have
$$\big (\frac{n(p-1)}{2(p+1)} - \frac 1{p+1}\big ) \|\bar w\|_{L^{p+1}(\R^n)}^{p+1}
+\frac 12 \|\bar w\|_{L^2(\R^n)}^{2}\leq {\mathcal E}_0.$$
As a conclusion we deduce that above we have equality everywhere
and hence the unique possibility is that $\bar \lambda=1$ and we conclude.

Next we show via a Schwartz symmetrization argument
that it is not restrictive to assume the minimizing sequence to be radially symmetric.
Hence given $w_k\in H^{1/2}(\R^n)$ that satisfies \eqref{minim} (but not necessarily radially symmetric), then we can construct another
radially symmetric sequence
$u_k\in H^{1/2}_{rad}(\R^n)$ that satsifies \eqref{minim}.
We introduce $w_k^*$ as the Schwartz symmetrization of $w_k$.
Notice that we have by standard facts about Schwartz symmetrization
that
\begin{equation}\label{limin}\liminf_{k\rightarrow \infty} {\mathcal E}_{hw} (w_k^*)\leq {\mathcal E}_0 \hbox{ and }
 \|w_k^*\|_{H^{1/2}(\R^n)}^2\leq \frac{d(p-1)}{p+1}\|w_k^*\|_{L^{p+1}(\R^n)}^{p+1}.
 \end{equation}
Notice that in principle ${\mathcal P}(w_k^*)\leq 0$.
On the other hands
$$\exists \lambda_k\in (0, 1] \hbox{ s.t. } {\mathcal P}(\lambda_k w_k^*)=0.$$
We conclude if we show that $\lambda_k\rightarrow 1$.
In fact in this case it is easy to check that
$\lambda_k w_k^*\in H^{1/2}_{rad}(\R^n)$, $\lambda_k w_k^*\in \mathcal M$
and ${\mathcal E}_{hw} (\lambda_k w_k^*)-{\mathcal E}_{hw}(w_k^*)
\rightarrow 0$ and hence
we conclude by \eqref{limin}
that ${\mathcal E}_{hw} (\lambda_k w_k^*)\rightarrow {\mathcal E}_0$.

In order to prove $\lambda_k\rightarrow 1$ we notice that by
\eqref{equivebergy} we have
\begin{align*}
{\mathcal E}_0\leq {\mathcal E}_{hw}(\lambda_k w_k^*)&=
\big (\frac{n(p-1)}{2(p+1)} - \frac 1{p+1}\big ) \lambda_k^{p+1}
\|
w_k^*\|_{p+1}^{p+1}
+\frac 12\lambda_k^2 \|w_k^*\|_2^2
\\\nonumber & \leq \big (\frac{n(p-1)}{2(p+1)} - \frac 1{p+1}\big ) \lambda_k^{2}
\|
w_k^*\|_{p+1}^{p+1}
+\frac 12\lambda_k^2 \|w_k\|_{L^2(\R^n)}^2
\\\nonumber&=  \big (\frac{n(p-1)}{2(p+1)} - \frac 1{p+1}\big ) \lambda_k^{2}
\|
w_k\|_{p+1}^{p+1}
+\frac 12\lambda_k^2 \|w_k\|_{L^2(\R^n)}^2
=\lambda_k^2 {\mathcal E}_{hw} (w_k)\end{align*}
where we used
\eqref{equivebergy} at the last step. We deduce that $\lambda_k\rightarrow 1$
since ${\mathcal E}_{hw} (w_k)\rightarrow {\mathcal E}_0$.

\end{proof}

We can now deduce for every $r>0$
the existence of solitary waves belonging to $S_r$ that moreover
are minimizers of ${\mathcal E}_{hw}$ constraint to $S_r \cap{\mathcal M}$.
In fact let $w$ be as in Lemma \ref{uncos}. Notice that by Lemma \ref{lagr}
we get
$$\sqrt{-\Delta} w + w-w^p=0.
$$
Moreover it is clear that $w\in {\mathcal A}_{r_0}$
where $r_0=\|w\|_{L^2(\R^n)}^2$.
Hence the first part of Theorem \ref{HWar}
is proved for $r=r_0$.
The case of a generic $r$ can be achieved by a straightforward rescaling argument.


\subsection{Inflation of $H^{1/2}$-norm for ${\mathcal P} (f)<0$
and ${\mathcal E}_{hw}(f)<{\mathcal I}_r$}

In this section we follow the approach of \cite{BHL}. In the sequel the
radially symmetric function $\varphi:\R^n \rightarrow \R$ is defined as
\begin{equation}\varphi ( r )=\left\{\begin{matrix} \label{AS}
\frac{r^2}{2} \ \text{ for } r \leq 1;   \\
const  \ \text{ for } r \geq 10.  \\
\end{matrix}\right.
\end{equation}
with $\varphi''(t)\leq 2$ for $r>0$, and we introduce the rescaled function $\varphi_R:\R^n\rightarrow \R$ as $\varphi_R(x):=R^2\varphi(\frac{x}{R})$.\\We define the localized virial in the spirit of Ogawa-Tsutsumi \cite{OZ}
\begin{equation}\label{def:locvirial}
M_{\varphi}(u)= 2 \text{ Im} \int_{\R^n} \bar u \nabla \varphi \cdot \nabla u dx
\end{equation}
In Lemma A.1 of \cite{BHL} it is shown that $M_{\varphi}(u)$ can be bounded,
as follows:
\begin{equation}\label{lem:boundloc}
|M_{\varphi}(u)|\leq C \left( \|u\|_{\dot H^{1/2}(\R^n)}^2+\|u\|_{L^2(\R^n)}\|u\|_{\dot H^{1/2}(\R^n)} \right)
\end{equation}
where the constant $C$ depends only on $\|\nabla \varphi\|_{W^{1, \infty}(\R^n)}$ and on the space dimension.
The following Lemma is crucial for our result.
\begin{lem}[Lemma 2.1, \cite{BHL}]\label{eq:idvir}Let $n\geq 2$, for any $f \in H^{1}_{rad}(\R^n)$ we have:
\begin{align}\label{eq:BHL}
&\frac{d}{dt}M_{\varphi}(u)=\\\nonumber\int_0^{\infty} m^\frac{1}{2}\left(\int_{\R^n} 4 \bar{\partial_k u_m}(\partial_{lk}^2\varphi)\partial_l u_m-(\Delta^2 \varphi)|u_m|^2dx\right)&dm-\frac{2(p-1)}{p+1} \int_{\R^n}(\Delta \varphi)|u|^{p+1} dx
\end{align}
where $u_m(t,x):=\sqrt{\frac{1}{\pi}}\mathcal{F}^{-1}\left(\frac{\hat u(t, \xi)}{\xi^2+m^2}\right)$
and $u(t,x)$ is the unique solution to HW with initial condition $u(0,x)=f(x)$.
\end{lem}
In the sequel we use the following Stein--Weiss inequality for radially symmetric functions due to Rubin \cite{R} in general space dimension $n$.
\begin{thm}[Rubin  \cite{R}]\label{thm:Rubin}
Let $n\ge 2$ and $0<s<n$. Then for all $u \in \dot H_{rad}^{s}(\R^n)$ we have:
\begin{equation}\label{DeNapoli}
\left(\int_{\R^n} |u(x)|^r |x|^{-\beta r} dx\right)^{1/r}\le C(n,s,r,\beta)\|u \|_{\dot{H}^{s}(\R^n)},
\end{equation}
where $r\ge 2$ and
\begin{equation}\label{1}
-(n-1)\big(\frac{1}{2}-\frac{1}{r}\big)\le\beta<\frac{n}{r},
\end{equation}
\begin{equation}\label{2}
\frac{1}{r}=\frac{1}{2}+\frac{\beta-s}{n}.
\end{equation}
\end{thm}
As a special case  of Rubin theorem when the dimension is $n\geq 2$,  $r=p+1$, $s=\frac 14$, $\beta=\frac{p(-2n+1)+2n+1}{4(p+1)}$ we have the following inequality
\begin{equation}\label{inegcruc}
\left(\int_{\R^n} |u (x)|^{p+1} |x|^{-\frac{p(-2n+1)+2n+1}{4}} dx\right)^\frac{1}{p+1}\le C \|u\|_{\dot H^{1/4}(\R^n)}
\end{equation}
that holds if $1 <p \leq 3$, which is satisfied since
we are assuming $1+\frac 2n <p<1+\frac 2{n-1}$.  As a byproduct of \eqref{inegcruc}
and by noticing that $\frac{p(-2n+1)+2n+1}{4}<0$  (this follows by the fact $p>1+\frac 2n$)
we have the following crucial decay:
\begin{align}\label{eq:crucdec}
\left(\int_{|x|\geq R} |u (x)|^{p+1}  dx\right) &\leq CR^{\frac{p(-2n+1)+2n+1}{4}}\|u\|_{\dot H^{1/4}(\R^n)}^{p+1}\\\nonumber &\leq CR^{\frac{p(-2n+1)+2n+1}{4}}\|u\|_{L^2(\R^n)}^{\frac{p+1}{2}}\|u\|_{\dot H^{1/2}(\R^n)}^{\frac{p+1}{2}}.
\end{align}

We shall also need the following result.

\begin{lem}\label{poho}
For every $\delta, r>0$ we have
$$\sup_{\{u\in S_r| {\mathcal P}(u)<0,
{\mathcal E}_{hw}(u)\leq {\mathcal I}_r-\delta\}}
{\mathcal P}(u)<0.$$
\end{lem}
\begin{proof}
Assume by the absurd that it is false, then for some $\delta_0, r_0>0$,
we get
$$\exists u_k\in S_{r_0} \hbox{ s.t. }  {\mathcal E}_{hw}(u_k)\leq {\mathcal I}_{r_0}-\delta_0,
{\mathcal P}(u_k)<0,
{\mathcal P}(u_k)\rightarrow 0.$$
As a first remark we get
\begin{equation}\label{sup}\sup_k \|u_k\|_{\dot H^{1/2}(\R^n)}<\infty.
\end{equation}
In fact it follows by
$$\frac{np-n-2}{2n(p-1)}\|u_k\|_{\dot H^{1/2}(\R^n)}^2={\mathcal E}_{hw}(u_k) - \frac 2{n(p-1)}
{\mathcal P}(u_k)$$
and we conclude since $\limsup$ of the r.h.s. is below ${\mathcal I}_{r_0}-\delta_0$.
Moreover we have
\begin{equation}\label{inf}\inf_k \|u_k\|_{\dot H^{1/2}(\R^n)}>0.\end{equation}
In fact notice that by assumption
\begin{align}\label{II}\|u_k\|_{\dot H^{1/2}(\R^n)}^2&= \frac{n(p-1)}{p+1}
\|u_k\|_{L^{p+1}(\R^n)}^{p+1}+{\mathcal P}(u_k)
\\\nonumber &\hbox{ where } {\mathcal P}(u_k)\rightarrow 0, {\mathcal P}(u_k)<0.
\end{align}
By combining this fact with the Gagliardo-Nirenberg inequality
and by recalling that $\|u_k\|_{L^2(\R^n)}^2=r_0>0$
we get for suitable universal constants $C_0, \epsilon_0>0$ (that depend from $p,r_0$)
$$\|u_k\|_{\dot H^{1/2}(\R^n)}^2\leq C_0 \|u_k\|_{\dot H^{1/2}(\R^n) }^{2+\epsilon_0},$$
and it
implies \eqref{inf}.
Notice also that by a simple continuity argument
\begin{equation}\label{rsical}
\exists \lambda_k \in (0, 1) \hbox{ s.t. } {\mathcal P}(\lambda_k u_k)=0.\end{equation}
We claim that $\lambda_k\rightarrow 1$. In fact we get by definition of ${\mathcal P}$
we get:
\begin{equation}\label{I}\lambda_k^2\|u_k\|_{\dot H^{1/2}(\R^n)}^2= \frac{n(p-1)}{p+1}
\lambda_k^{p+1} \|u_k\|_{L^{p+1}(\R^n)}^{p+1}.\end{equation}
By combining the identities \eqref{I} and \eqref{II} above we get
$$\frac{n(p-1)}{p+1}
(\lambda_k^{p-1} -1) \|u_k\|_{L^{p+1}(\R^n)}^{p+1}={\mathcal P}(u_k)\rightarrow 0.$$
We conclude that $\lambda_k\rightarrow 1$
provided that we show $\inf_k \|u_k\|_{L^{p+1}(\R^n)}>0$. Of course it is true
otherwise we get
\begin{equation}\label{limit}0=\lim_{k\rightarrow \infty}
{\mathcal P}(u_k)=\lim_{k\rightarrow \infty} \frac 12
\|u_k\|_{\dot H^{1/2}(\R^n)}^2\end{equation}
which is in contradiction with \eqref{inf}.
As a consequence of the fact $\lambda_k\rightarrow 1$ we deduce
$\|\lambda_k u_k\|_{2}^2=r_k\rightarrow r_0$
and hence by \eqref{rsical} we get
${\mathcal E}_{hw}(\lambda_k u_k) \geq {\mathcal I}_{r_k}\rightarrow
{\mathcal I}_{r_0}$
(the last limit follows by elementary considerations). In particular we get
${\mathcal E}_{hw}(\lambda_k u_k) \geq {\mathcal I}_{r_0}-\frac{\delta_0}2$.
Moreover
\begin{align*}&{\mathcal E}_{hw}(u_k) - {\mathcal E}_{hw}(\lambda_k u_k)\\\nonumber
=\frac 12 (1-\lambda_k^2)\|u_k\|_{\dot
H^{1/2}(\R^n)}^2 &- \frac{1}{p+1}(1- \lambda_k^{p+1}) \|u_k\|_{L^{p+1}(\R^n)}^{p+1}\rightarrow 0\end{align*}
where we used \eqref{sup} with $\lambda_k\rightarrow 1$.
We get a
contradiction since
 ${\mathcal E}_{hw}(\lambda_k u_k) \geq {\mathcal I}_{r_0}-\frac{\delta_0}2$ and ${\mathcal E}_{hw}(u_k) \leq {\mathcal I}_{r_0}-\delta_0.$

\end{proof}

\begin{lem}[Localized virial identity for HW]\label{lem:ass}
There exists a constant $C>0$ such that
\begin{equation}\label{bridg}\frac{d}{dt}M_{\varphi_R}(u)\leq 4 {\mathcal P}(u) +C(R^{-1}+R^{\frac{p(-2n+1)+2n+1}{4}}\|u\|_{L^2(\R^n)}^{\frac{p+1}{2}}\|u\|_{\dot H^{1/2}(\R^n)}^{\frac{p+1}{2}})\end{equation}
for any $u(t,x)$ radially symmetric solution to HW.
\end{lem}
\begin{proof}
In \cite{BHL} it is shown (by choosing $s=\frac 12$) the following estimates:
$$4 \int_0^{\infty} m^{1/2} \int_{\R^n} {\partial_k \bar u_m}(\partial_{lk}\varphi_R)\partial_l u_m dx dm\leq 2 \|u(t)\|_{\dot H^{1/2}(\R^n)}^2$$
and
$$\left|  \int_0^{\infty} m^{1/2} \int_{\R^n}(\Delta^2 \varphi_R)|u_m|^2dx dm\right|\leq C R^{-1}.$$
Concerning the last term in \eqref{eq:BHL}, namely  $-\frac{2(p-1)}{p+1} \int_{\R^n}(\Delta \varphi_R)|u|^{p+1} dx$, we have
\begin{align*}-\frac{2(p-1)}{p+1} \int_{\R^n}(\Delta \varphi_R)|u|^{p+1} dx&\\\nonumber =-\frac{2n(p-1)}{p+1} \int_{\R^n}
|u|^{p+1} dx& -\frac{2(p-1)}{p+1} \int_{\R^n }(\Delta \varphi_R-n )|u|^{p+1} dx.\end{align*}
Notice that $\Delta \varphi_R=n $ on $\left\{|x|\leq R\right\}$, hence by recalling \eqref{eq:crucdec} and summarizing the estimates above we get:
\begin{align*}\frac{d}{dt}M_{\varphi_R}(u)&\leq 2\|u\|_{\dot H^{1/2}(\R^n)}^2-\frac{2n(p-1)}{p+1} \int_{\R^n }|u|^{p+1} dx\\\nonumber& +C(R^{-1}+R^{\frac{p(-2n+1)+2n+1}{4}}\|u\|_{L^2(\R^n)}^{\frac{p+1}{2}}\|u\|_{\dot H^{1/2}(\R^n)}^{\frac{p+1}{2}})\end{align*}
which is equivalent to
$$\frac{d}{dt}M_{\varphi_R}(u)\leq 4 {\mathcal P}(u) +C(R^{-1}+R^{\frac{p(-2n+1)+2n+1}{4}}\|u\|_{L^2(\R^n)}^{\frac{p+1}{2}}\|u\|_{\dot H^{1/2}(\R^n)}^{\frac{p+1}{2}}).$$
\end{proof}

We can now conclude the proof on the inflation of the norms, in the case that the solution
exists globally in time.
First notice that by Lemma \ref{poho} we get
 ${\mathcal P}(u(t,x))<-\delta<0$, and we claim that it implies
 \begin{equation}\label{eq:stimasympt2}
\inf_t \|u(t,x)\|_{\dot H^{1/2}(\R^n)}>0
 \end{equation}
and for $R$ sufficiently large
  \begin{equation}\label{eq:stimasympt}
 \frac{d}{dt}M_{\varphi_R}(u)\leq 2 {\mathcal P}(u) -\alpha\|u\|_{\dot H^{1/2}(\R^n)}^2< -\alpha \|u\|_{\dot H^{1/2}(\R^n)}^2.
 \end{equation}
for some constant $\alpha>0$.
Let us first prove \eqref{eq:stimasympt2} and assume by contradiction the existence of a sequence of times $t_n$  such that $\lim_{n \rightarrow \infty} \|u(t_n,x)\|_{\dot H^{1/2}}^2=0$. This fact implies by Sobolev embedding and conservation of the mass, that ${\mathcal P}(u(t_n,x))\rightarrow 0$ which contradicts ${\mathcal P}(u(t,x))<-\delta<0$.\\
Now  let us  prove \eqref{eq:stimasympt}. By using \eqref{bridg}
it is sufficient to prove
that there exists $\alpha$ sufficiently small such that
$$ 4 {\mathcal P}(u) + C(R^{-1}+ R^{\frac{p(-2n+1)+2n+1}{4}}\|u\|_{L^2(\R^n)}^{\frac{p+1}{2}}\|u\|_{\dot H^{1/2}(\R^n)}^{\frac{p+1}{2}}) \leq 2 {\mathcal P}(u) -\alpha \|u\|_{\dot H^{1/2}(\R^n)}^2$$
 Notice  that $\frac{p+1}{2}<2$ and thanks to  \eqref{eq:stimasympt2} and conservation of the mass we have the inequality
 $$4 {\mathcal P}(u) + C(R^{-1}+R^{\frac{p(-2n+1)+2n+1}{4}}\|u\|_{L^2(\R^n)}^{\frac{p+1}{2}}\|u\|_{\dot H^{1/2}(\R^n)}^{\frac{p+1}{2}})$$$$\leq 4 {\mathcal P}(u) + C(R^{-1}+R^{\frac{p(-2n+1)+2n+1}{4}}\|u\|_{\dot H^{1/2}(\R^n)}^{2})$$
and hence it suffices to show that
 \begin{equation}\label{eq:kmot}
  2 {\mathcal P}(u) + C(R^{-1}+R^{\frac{p(-2n+1)+2n+1}{4}}\|u\|_{\dot H^{1/2}(\R^n)}^{2})+\alpha \|u\|_{\dot H^{1/2}(\R^n)}^2<0
  \end{equation}
 to get \eqref{eq:stimasympt}.\\
 From the identity
 $${\mathcal E}_{hw}(u)-\frac{2}{n(p-1)}{\mathcal P}(u)=\frac{r}{2}+\frac{n(p-1)-2}{2n(p-1)}\|u\|_{\dot H^{1/2}(\R^n)}^2$$
 we get
 $$\|u\|_{\dot H^{1/2}(\R^n)}^2\leq \frac{2n(p-1)}{n(p-1)-2}{\mathcal E}_{hw}(u)-\frac{4}{n(p-1)-2}
 {\mathcal P}(u).$$
 As a consequence we get
 \begin{align*}&2 {\mathcal P}(u) + C(R^{-1}+R^{\frac{p(-2n+1)+2n+1}{4}}\|u\|_{\dot H^{1/2}(\R^n)}^{2})+\alpha \|u\|_{\dot H^{1/2}(\R^n)}^2\\\nonumber &<
\big  (2-\frac{4C}{n(p-1)-2}R^{\frac{p(-2n+1)+2n+1}{4}}-\frac{4\alpha}{n(p-1)-2}\big )
 {\mathcal P}(u)
\\\nonumber& +\frac{2nC(p-1)}{n(p-1)-2} R^{\frac{p(-2n+1)+2n+1}{4}}{\mathcal E}_{hw}(f)
 + \frac{2n\alpha (p-1)}{n(p-1)-2}{\mathcal E}_{hw}(f) + CR^{-1}
  \end{align*}
 Notice that we can conclude \eqref{eq:kmot} since ${\mathcal P}(u(t,x))<-\delta$
and hence it is sufficient
to select $\alpha$ very small and $R$ very large.
 By combining
 \eqref{eq:stimasympt2} with
 \eqref{lem:boundloc}
we get
 \begin{equation}\label{lem:boundloc2}
|M_{\varphi_R}(u)|\leq C( R ) \|u\|_{\dot H^{1/2}(\R^n)}^2.
\end{equation}
From \eqref{eq:stimasympt} we deduce that we can select $t_1\in \R$ such that $M_{\varphi_R}(u(t))\leq 0$ for $t\geq t_1$ and $M_{\varphi_R}(u(t_1))=0$.
Hence integrating \eqref{eq:stimasympt} we obtain
$$M_{\varphi_R}(u(t))\leq -\alpha  \int_{t_1}^{t}\|u(s)\|_{\dot H^{1/2}(\R^n)}^2 ds.$$
Now by using \eqref{lem:boundloc2}
we get the  integral inequality
$$\left| M_{\varphi_R}(u(t))\right| \geq C(R, \alpha ) \int_{t_1}^{t} \left| M_{\varphi_R}(u(s))\right|  ds,$$
which yields an exponential lower bound.

\subsection{Instability of ${\mathcal A}_r$}

Given $v(x)\in {\mathcal A}_{r}$  we shall show that there exists
a sequence $\lambda_k\rightarrow 1$, $\lambda_k>1$
such that
$${\mathcal P}(\lambda_k^{d/2} v(\lambda_k x))<0, \quad {\mathcal E}_{hw}(\lambda_k^{d/2} v(\lambda_k x))<{\mathcal E}_{hw}(v)={\mathcal I}_r.$$
Then by denoting with $v_k(t,x)$ the unique solution
to HW such that $v_k(0,x)=\lambda_k^{d/2} v(\lambda_k x)$ we get, by the inflation of norm proved in the previous subsection,
that $v_k(t,x)$ are unbounded
in $H^{1/2}$ for large time, despite to the fact that they are arbitrary close to $v(x)$ at the initial time $t=0$. Of course it implies the instability of ${\mathcal A}_r$.\\
In order to prove the existence of $\lambda_k$ as above, we introduce the functions
$$h:(0, \infty)\ni \lambda \rightarrow {\mathcal P} (\lambda^{d/2} v(\lambda x))=\frac 12 \lambda \|v\|_{\dot H^{1/2}(\R^n)}^2- \frac{n(p-1)}{2(p+1)} \lambda^{n (p-1)/2} \| v \|^{p+1}_{L^{p+1}(\R^n)},$$
\begin{align*}g:(0, \infty)\ni \lambda &\rightarrow {\mathcal E}_{hw} (\lambda^{n/2} v(\lambda x)) \\\nonumber &=\frac{1}{2} \lambda \|v\|_{\dot H^{1/2}(\R^n)}^2 +\|v\|^2_{L^2(\R^n)}
-\frac{1}{p+1} \lambda^{n(p-1)/2}\| v \|^{p+1}_{L^{p+1}(\R^n)}.\end{align*}
Notice that since $v(x)\in \mathcal M$ we get $h(1)=0$ and hence
by elementary analysis of the function $h$ we deduce that $h(\lambda)<0$ for every $ \lambda>1$. Moreover again by the fact that $v(x)\in {\mathcal M}$
we get $g'(1)=0$ and since $g(1)= {\mathcal E}_{hw}(v)={\mathcal I}_r$ we deduce that
$g(\lambda)<{\mathcal I}_r$ for every $\lambda>0$.
It is now easy to conclude the existence of $\lambda_k$ with the desired property.

\section{Appendix}

In this appendix we prove a global existence result for 1-d quartic
sNLS with initial condition in $H^{3/2}(\R)$.
It is worth mentioning  that the same argument used
in \cite{OV}, where it is treated 1-d quartic HW, can be adapted to sNLS.
Hence
the result stated below can be improved by assuming $f(x)\in H^1(\R)$.
However we want to give the argument below, since we believe that it is more transparent and  slightly simpler. In particular it does not involve the use of fractional Leibnitz rules as in \cite{OV}. Moreover, in our opinion, it makes more clear the argument that stands behind
the modified energy technique, which is a basic tool in \cite{OV} and that hopefully will be  a basic tool to deal with other situations (see for instance \cite{PTV}).

\begin{thm} Let us fix $n=1$ and $p=4$.
Assume that $u(t,x)$ solves sNLS with initial datum
$f(x)\in H^{3/2}(\R)$ and assume moreover that $$\sup_{(-T_-(f), T_+(f))} \|u(t, x)\|_{H^{1/2}(\R)}<\infty,$$ where $(-T_-(f), T_+(f))$ is the maximal interval of existence.
Then necessarily $T_\pm (f)=\infty$, namely the solution is global.
\end{thm}

\begin{proof}
We have to show that the norm $\|u(t,x)\|_{H^{3/2}}$ cannot blow-up in finite time.
In order to do that we notice that
$$\im \partial_t (\partial_x u)=(\sqrt {1-\Delta}) \partial_x u- \partial_x (u|u|^3)$$
and hence if we
multiply this equation by $\partial_t (\partial_x \bar u)$, we integrate by parts
and we get the real part, then we obtain:
$$0=\Re \int_\R (\sqrt {1-\Delta}) \partial_x u \partial_t (\partial_x \bar u) dx -
\Re \int_\R \partial_t (\partial_x \bar u) \partial_x (u|u|^3) dx
$$
$$= \Re \int_\R (1-\Delta)^{1/4} \partial_x u \partial_t (\partial_x (1-\Delta)^{1/4} \bar u) dx
- \Re \int_\R \partial_t (\partial_x \bar u) \partial_x u|u|^3 dx- \Re \int_\R \partial_t (\partial_x \bar u) u \partial_x |u|^3 dx$$
and hence
$$0= \frac 12 \frac d{dt} \|(1-\Delta)^{1/4} \partial_x u\|_{L^2(\R)}^2
-\frac 12 \int_\R \partial_t (|\partial_x u|^2)|u|^3 dx - \frac 32 \Re \int_\R \partial_t (\partial_x \bar u) u \partial_x |u|^2 |u| dx $$
$$= \frac 12 \frac d{dt} \|(1-\Delta)^{1/4} \partial_x u\|_{L^2(\R)}^2
-\frac 12 \frac d{dt} \int_\R |\partial_x u|^2|u|^3 dx +\frac 12 \int_\R |\partial_x u|^2
\partial_t (|u|^3) dx $$$$
- \frac 32 \Re \int_\R \partial_t (\partial_x \bar u) u \partial_x \bar u u |u| dx-
\frac 32 \Re \int_\R \partial_t (\partial_x \bar u) \partial_x u |u|^3 dx.$$
We can continue as follows
$$0= \frac 12 \frac d{dt} \|(1-\Delta)^{1/4} \partial_x u\|_{L^2(\R)}^2
-\frac 12 \frac d{dt} \int_\R |\partial_x u|^2|u|^3 dx +\frac 12 \int_\R |\partial_x u|^2
\partial_t (|u|^3) dx $$$$
- \frac 34 \Re \int_\R \partial_t [(\partial_x \bar u)]^2 u^2 |u| dx-
\frac 34 \Re \int_\R \partial_t (|\partial_x  u|^2) |u|^3 dx$$
$$= \frac 12 \frac d{dt} \|(1-\Delta)^{1/4} \partial_x u\|_{L^2(\R)}^2
-\frac 12 \frac d{dt} \int_\R |\partial_x u|^2|u|^3 dx +\frac 12 \int_\R |\partial_x u|^2
\partial_t (|u|^3) dx $$$$
- \frac 34 \frac d{dt} \Re \int_\R (\partial_x \bar u)^2 u^2 |u| dx
+ \frac 34 \Re \int_\R (\partial_x \bar u)^2 \partial_t (u^2 |u|) dx
$$$$-
\frac 34 \frac d{dt} \Re \int_\R |\partial_x  u|^2 |u|^3 dx
+\frac 34 \Re \int_\R |\partial_x  u|^2 \partial_t (|u|^3) dx.
$$
Summarizing we get
\begin{equation}\label{ident}\frac d{dt} {\mathcal E} (u)=
-\frac 54 \int_\R |\partial_x u|^2
\partial_t (|u|^3) dx
- \frac 34 \Re \int_\R (\partial_x \bar u)^2 \partial_t (u^2 |u|) dx
.
\end{equation}
where
$${\mathcal E} (u)=
\frac 12 \|(1-\Delta)^{1/4} \partial_x u\|_{L^2(\R)}^2
-\frac 54  \int_\R |\partial_x u|^2|u|^3 dx
- \frac 34 \Re \int_\R (\partial_x \bar u)^2 u^2 |u| dx
.
$$
Notice that since $u(t,x)$ solves sNLS we deduce that the r.h.s. in \eqref{ident}
can be estimated by the following quantity (up to a constant):
\begin{align*}\int_{\R} |\partial_x u|^2 &|\sqrt{1-\Delta} u| |u|^2 dx +
\int_{\R} |\partial_x u|^2 |u|^6 dx
\\\nonumber & \leq \|u\|_{W^{1,4}(\R)}^2 \|u\|_{H^1(\R)}\|u\|_{L^\infty(\R)}^2 + \|u\|_{H^1(\R)}^2
\|u\|_{L^\infty(\R)}^6.
\end{align*}
Next by using the Gagliardo-Nirenberg inequality
$\|u\|_{W^{1,4}(\R)}^2\leq C \|u\|_{H^{3/2}(\R)} \|u\|_{H^{1}(\R)}$,
and the Brezis-Gallou\"et inequality
(see \cite{BrGa}),
together with the assumption on the boundedness of $H^{1/2}$-norm of $u(t,x)$,
we can continue the estimate above as follows:
\begin{align}& \label{rhs} \hbox { r.h.s. \eqref{ident} }\leq C \|u\|_{H^{3/2}(\R)} \|u\|_{H^{1}(\R)}^2 \ln (2+ \|u\|_{H^{3/2}(\R)})
\\\nonumber&+ \|u\|_{H^{3/2}(\R)} \ln^3 (2+ \|u\|_{H^{3/2}(\R)})
\\\nonumber&\leq C \|u\|_{H^{3/2}(\R)}^2 \ln (2+ \|u\|_{H^{3/2}(\R)})
+ \|u\|_{H^{3/2}(\R)} \ln^3 (2+ \|u\|_{H^{3/2}(\R)}).
\end{align}
Notice also that the second and third term involved in the energy ${\mathcal E}(u)$
can be estimated as follows:
$$\|u\|_{H^1(\R)}^2 \|u\|_{L^\infty(\R)}^3\leq C \|u\|_{H^{3/2}(\R)}
\ln^{3/2} (2+ \|u\|_{H^{3/2}(\R)}).$$
By combining the estimate above with \eqref{rhs} and \eqref{ident}, and by recalling that $\frac d{dt} \|u\|_{L^2(\R)}^2=0$ we deduce
\begin{align*}&\|u(t)\|_{H^{3/2}(\R)}^2  \leq C \|u(0)\|_{H^{3/2}(\R)}^2 +
C \sup_{s\in (0,t)}\|u(s)\|_{H^{3/2}(\R)}
\ln^{3/2} (2+ \|u(s)\|_{H^{3/2}(\R)})
\\\nonumber & + C \int_0^t  \|u(s)\|_{H^{3/2}(\R)}^2 \ln (2+ \|u(s)\|_{H^{3/2}(\R)})
 + \|u(s)\|_{H^{3/2}(\R)} \ln^3 (2+ \|u(s)\|_{H^{3/2}(\R)}) ds.\end{align*}
We conclude by a suitable version of the Gronwall Lemma
(see \cite{OV} for more details on this point).
\end{proof}

\end{document}